\documentclass[12pt]{amsart}       
\usepackage{txfonts}
\usepackage{extarrows}
\usepackage{amssymb}
\usepackage{eucal}
\usepackage{bbm}
\usepackage{graphicx}
\usepackage{amsmath}
\usepackage{amscd}
\usepackage[all]{xy}           
\usepackage{amsfonts,latexsym}
\usepackage{xspace}
\usepackage{epsfig}
\usepackage{float}
\usepackage{color}
\usepackage{shuffle}
\usepackage{fancybox}
\usepackage{colordvi}
\usepackage{multicol}
\usepackage{colordvi}
\usepackage{ifpdf}
\ifpdf
  \usepackage[colorlinks,final,backref=page,hyperindex]{hyperref}
\else
  \usepackage[colorlinks,final,backref=page,hyperindex,hypertex]{hyperref}
\fi
\usepackage[active]{srcltx} 
\usepackage{tikz}
\usepackage{graphicx}

\newtheorem{theorem}{Theorem}[section]
\newtheorem{proposition}[theorem]{Proposition}
\newtheorem{lemma}[theorem]{Lemma}
\newtheorem{coro}[theorem]{Corollary}
\newtheorem{prop-def}{Proposition-Definition}[section]
\newtheorem{coro-def}{Corollary-Definition}[section]

\theoremstyle{definition}
\newtheorem{definition}[theorem]{Definition}
\newtheorem{remark}[theorem]{Remark}

\newcommand{\nc}{\newcommand}
\nc{\tred}[1]{\textcolor{red}{#1}}
\nc{\tblue}[1]{\textcolor{blue}{#1}}
\nc{\tgreen}[1]{\textcolor{green}{#1}}
\nc{\tpurple}[1]{\textcolor{purple}{#1}}
\nc{\btred}[1]{\textcolor{red}{\bf #1}}
\nc{\btblue}[1]{\textcolor{blue}{\bf #1}}
\nc{\btgreen}[1]{\textcolor{green}{\bf #1}}
\nc{\btpurple}[1]{\textcolor{purple}{\bf #1}}
\nc{\NN}{{\mathbb N}}
\nc{\ncsha}{{\mbox{\cyr X}^{\mathrm NC}}} \nc{\ncshao}{{\mbox{\cyr
X}^{\mathrm NC}_0}}

\newcommand{\delete}[1]{}

\nc{\mlabel}[1]{\label{#1}}
\nc{\mcite}[1]{\cite{#1}}
\nc{\mref}[1]{\ref{#1}}
\nc{\meqref}[1]{\eqref{#1}}
\nc{\mbibitem}[1]{\bibitem{#1}}

\delete{ Use the next lines to show names
\nc{\mlabel}[1]{\label{#1}{\hfill \hspace{1cm}{\bf{{\ }\hfill(#1)}}}}
\nc{\mcite}[1]{\cite{#1}{{\bf{{\ }(#1)}}}}
\nc{\mref}[1]{\ref{#1}{{\bf{{\ }(#1)}}}}
\nc{\meqref}[1]{\eqref{#1}{{\bf{{\ }(#1)}}}}
\nc{\mbibitem}[1]{\bibitem[\bf #1]{#1}}
}
\nc{\sha}{{\mbox{\cyr X}}}  
\newfont{\scyr}{wncyr10 scaled 550}
\nc{\ssha}{\mbox{\bf \scyr X}}
\nc{\shap}{{\mbox{\cyrs X}}} 
\nc{\shpr}{\diamond}    
\nc{\shp}{\ast} \nc{\shplus}{\shpr^+}
\nc{\shprc}{\shpr_c}    
\nc{\dep}{\mrm{dep}} \nc{\lc}{\lfloor} \nc{\rc}{\rfloor}
\nc{\db}{\leq_{\rm db}} \nc{\bfk}{{\bf k}}

\nc{\cala}{{\mathcal A}} \nc{\calb}{{\mathcal B}}
\nc{\calc}{{\mathcal C}}
\nc{\cald}{{\mathcal D}} \nc{\cale}{{\mathcal E}}
\nc{\calf}{{\mathcal F}} \nc{\calg}{{\mathcal G}}
\nc{\calh}{{\mathcal H}} \nc{\cali}{{\mathcal I}}
\nc{\call}{{\mathcal L}} \nc{\calm}{{\mathcal M}}
\nc{\caln}{{\mathcal N}} \nc{\calo}{{\mathcal O}}
\nc{\calp}{{\mathcal P}} \nc{\calr}{{\mathcal R}}
\nc{\cals}{{\mathcal S}} \nc{\calt}{{\mathcal T}}
\nc{\calu}{{\mathcal U}} \nc{\calw}{{\mathcal W}} \nc{\calk}{{\mathcal K}}
\nc{\calx}{{\mathcal X}} \nc{\CA}{\mathcal{A}}

\nc{\fraka}{{\mathfrak a}} \nc{\frakA}{{\mathfrak A}}
\nc{\frakb}{{\mathfrak b}} \nc{\frakB}{{\mathfrak B}}
\nc{\frakc}{{\mathfrak c}}
\nc{\frakD}{{\mathfrak D}} \nc{\frakF}{\mathfrak{F}}
\nc{\frakf}{{\mathfrak f}} \nc{\frakg}{{\mathfrak g}}
\nc{\frakH}{{\mathfrak H}} \nc{\frakL}{{\mathfrak L}}
\nc{\frakM}{{\mathfrak M}} \nc{\bfrakM}{\overline{\frakM}}
\nc{\frakm}{{\mathfrak m}} \nc{\frakP}{{\mathfrak P}}
\nc{\frakN}{{\mathfrak N}} \nc{\frakp}{{\mathfrak p}}
\nc{\frakS}{{\mathfrak S}} \nc{\frakT}{\mathfrak{T}}
\nc{\frakX}{{\mathfrak X}}

\font\cyr=wncyr10 \font\cyrs=wncyr7
\nc{\li}[1]{\textcolor{blue}{Nan:#1}}
\nc{\lir}[1]{\textcolor{red}{Li:#1}}
\nc{\yi}[1]{\textcolor{blue}{Yi: #1}}
\nc{\xing}[1]{\textcolor{purple}{Xing:#1}}
\nc{\revise}[1]{\textcolor{red}{#1}}

\numberwithin{equation}{section}
\nc{\etree}{1}
\nc{\RP}{{\mathcal{D}}^{\alpha}_{w}([0, T]^2, \RR^d)}
\nc{\RRP}{{\mathcal{D}}^{\alpha}([0, T]^2, \RR^d)}
\nc{\Y}{{\bf Y}}
\nc{\x}{\mathbb{X}}
\nc{\xx}{\mathcal{X}}
\nc{\HA}{\mathbb{H}}
\nc{\ha}{\mathcal{H}}
\nc{\Ha}{{\bf H}}

\nc{\RR}{\mathbb{R}} \nc{\ZZ}{\mathbb{Z}} \nc{\V}{\RR^{d}} \nc{\pro}{\otimes}
\nc{\tng}{T^{\le N}(\RR^d)^{g}} \nc{\tn}{T^{\le N}(\RR^d)}
\nc{\ttg}{T^{\le 3}(\RR^d)^{g}}
\nc{\X}{{\bf X}} \nc{\Z}{{\bf Z}}
\nc{\E}{{\bf E}}
\nc{\DA}{\mathcal{D}^{\alpha}_{w}([0,2\pi]^2, \RR^d)}
\nc{\C}{\mathcal{C}^{\alpha}}
\nc{\D}{\mathcal{D}^{\alpha}([0,T]^2, \RR^d)}
\nc{\CC}{\mathcal{C}_{\X}^{\alpha}}
\nc{\f}{\varphi}
\nc{\al}{\alpha}
\nc{\lbar}{\overline}
\nc{\Hh}{{\mathfrak{H}}}
\nc{\id}{\text{id}} \nc{\Id}{\text{Id}}
\nc{\bx}{(\delta X, \x, \xx)}
\nc{\bh}{(1, H_t^1, \ldots, H_t^N)}

\begin{document}
\title[Rough Burger-like SPDE\MakeLowercase{s}]{Rough Burger-like SPDE\MakeLowercase{s}}

%
\author{Nannan Li}
\address{School of Mathematics and Statistics, Lanzhou University
Lanzhou, 730000, China
}
\email{linn2024@lzu.edu.cn}

\author{Xing Gao$^{*}$}\thanks{*Corresponding author}
\address{School of Mathematics and Statistics, Lanzhou University
Lanzhou, 730000, China;
Gansu Provincial Research Center for Basic Disciplines of Mathematics
and Statistics, Lanzhou, 730070, China
}
\email{gaoxing@lzu.edu.cn}


\begin{abstract}
We study a class of nonlinear Burgers-type stochastic partial differential equations driven by additive space-time white noise in one spatial dimension. Building on the rough path framework initiated by Hairer, which provides a pathwise solution theory under spatial regularity $\alpha \in(\frac{1}{3}, \frac{1}{2})$, we extend this approach to the full subcritical regime $\alpha \in(0, \frac{1}{2})$.
Our main contribution is the establishment of pathwise existence and uniqueness of mild (equivalently, weak) solutions when the spatial regularity of the solution lies strictly below the classical rough path threshold. This is achieved through refined estimates for controlled rough paths, including a new upper bound for compositions with smooth functions and a scaling analysis for rough integrals against heat kernels. In particular, we extend and sharpen key analytic estimates originating from Hairer's work, incorporating refined scaling arguments that are effective in the low-regularity regime. As a result, our framework significantly enlarges the class of Burgers-type SPDEs that can be treated pathwise using rough path techniques.
\end{abstract}

\makeatletter
\@namedef{subjclassname@2020}{\textup{2020} Mathematics Subject Classification}
\makeatother
\subjclass[2020]{
60L20, 
60L50, 
60H15, 
60H17, 
}

\keywords{Burgers-like equation, stochastic partial differential equation, rough path}

\maketitle

\tableofcontents

\setcounter{section}{0}

\allowdisplaybreaks

\section{Introduction}
In this paper, we establish the existence and uniqueness of mild solutions to a class of nonlinear SPDEs of Burgers type. Our approach is developed within the framework of rough path theory, allowing us to treat the equations in a purely pathwise sense. The analysis is carried out under the assumption that the spatial regularity of the solution lies within the subcritical regime, specifically in the range $(0, \frac{1}{2})$.

\subsection{Burgers-like SPDEs}
SPDEs of Burgers type form a significant class of models in the study of nonlinear systems influenced by randomness. These equations naturally arise in a variety of scientific and engineering disciplines, including turbulence modeling, traffic flow, interface growth, and statistical mechanics. A canonical form of such an SPDE is given by
\begin{equation}
du = \left[ \partial_x^2 u + f(u) + g(u)\,\partial_x u \right]\,dt + \eta\,dW(t).
\label{eq:bspde}
\end{equation}
Here, $\eta > 0$ denotes the noise amplitude, and $(W_t)_{t\in [0,1]}$ is a standard cylindrical Wiener process on the space $L^2([0,2\pi], \mathbb{R}^d)$~\mcite{DZ92}, modeling space-time white noise. The solution $u : [0,1] \times [0,2\pi] \to \mathbb{R}^d$ is a random field subject to periodic boundary conditions in space. The nonlinear terms
\[
f: \mathbb{R}^d \to \mathbb{R}^d, \quad g: \mathbb{R}^d \to \mathcal{L}(\mathbb{R}^d, \mathbb{R}^d)
\]
are assumed to be smooth with all derivatives bounded.
Equation~\eqref{eq:bspde} can be viewed as a stochastic perturbation of a generalized form of the classical viscous Burgers equation,
a fundamental model in fluid dynamics and nonlinear wave phenomena. The addition of stochastic forcing not only introduces rich probabilistic behavior but also poses significant analytical challenges, particularly due to the irregularity induced by the noise.

One of the core difficulties in the analysis of~\eqref{eq:bspde} stems from the low spatial regularity of its solutions. This issue is already evident in the linear stochastic heat equation
\begin{equation}
dh = \partial_x^2 h\,dt + \eta\,dW(t),
\mlabel{eq:heat}
\end{equation}
whose solution $h$ is almost surely nowhere differentiable in space. As shown in~\cite{Wal86}, $h$ is almost surely $\alpha$-H\"older continuous for any $\alpha < \tfrac{1}{2}$, but not more in the sense that it is almost surely not $\tfrac{1}{2}$-H\"older continuous.  This regularity barrier prevents a direct interpretation of nonlinear expressions like $g(u)\,\partial_x u$ in the classical sense.
Earlier attempts to address this challenge relied on assuming that the function $g$ is a Jacobian, i.e., $g = DG$ for some smooth potential $G : \mathbb{R}^d \to \mathbb{R}^d$. Under this assumption, integration by parts can be used to define the nonlinear term in a weak sense~\cite{BCJ94, PDT94}. However, this requirement imposes a restrictive structural constraint that is undesirable, especially in higher-dimensional.
A major advance in overcoming this obstacle was provided by Martin Hairer in~\cite{Hair11}, who employed rough path theory to define and analyze the problematic nonlinear term. His key insight was to reinterpret the product $g(u)\,\partial_x u$ as a pathwise rough integral
\vspace{-0.05in}
\[
\int_0^{2\pi} \varphi(x)\,g(u(x))\,du(x),
\]
where $\varphi$ is a smooth test function. When $u$ is H\"older continuous with exponent $\alpha > \tfrac{1}{2}$, this integral can be defined using Young integration~\cite{You36}. However, when $\alpha \in (0, \tfrac{1}{2})$, Young's theory no longer applies, and one must appeal to the rough integral in terms of rough path theory~\cite{FV10b, Gu04, Ly98}.

In this article, we extend the conclusion of~\cite{Hair11}, developing a rough path framework capable of handling solutions with even broader spatial regularity, specifically $\alpha \in (0, \frac{1}{2})$. This requires incorporating additional stochastic cancellation effects into the analytic framework, which are not captured by classical techniques.
This development opens the door to analyzing a wider class of SPDEs with highly singular structure, and contributes to the ongoing effort of rigorously understanding nonlinear stochastic systems beyond the scope of classical techniques.
Our approach is based on lifting the solution $h$ of the linear SPDE~\eqref{eq:heat} to a Gaussian rough path. The well-posedness of this construction is supported by deep results in the theory of Gaussian rough paths~\cite{FV10a, FV10b, GOS20}, ensuring that such a lift is both meaningful and robust. This lifted path serves as a reference rough path against which nonlinearities such as $g(u)\,\partial_x u$ can be interpreted pathwise, even when $u$ itself lacks sufficient classical regularity.

\subsection{Rough paths} \mlabel{ss:rpath}
Rough path theory, introduced by Lyons~\cite{Ly98}, provides a framework for analyzing differential equations driven by signals of low regularity, particularly when classical integration theories such as It\^o or Young integration are no longer applicable. The central object of study is the rough differential equation (RDE)
\begin{equation}
\left\{
\begin{array}{ll}
dY_t &= f(Y_t)\cdot dX_t = \sum_{i=1}^{d} f_i(Y_t)\,dX_t^i, \quad t \in [0,T], \\
Y_0 &= \xi,
\end{array}
\right.
\label{eq:rde}
\end{equation}
where $X = (X^1, \ldots, X^d) : [0,T] \to \mathbb{R}^d$ is a path of H\"older regularity $\alpha \in (0,1]$, and $f_i : \mathbb{R}^m \to \mathbb{R}^m$ are smooth vector fields.
The key difficulty in interpreting~\eqref{eq:rde} lies in giving a precise meaning to the integral with respect to the rough signal $X$. This is resolved by lifting $X$ to a rough path $\mathbf{X}$, which enriches $X$ with iterated integrals that satisfy Chen's relation and H\"older-type bounds on the tensor algebra. Instead of treating $X$ as the sole driver, the RDE is reformulated as
\begin{equation}
\left\{
\begin{array}{ll}
d\mathbf{Y}_t &= f(\mathbf{Y}_t)\cdot d\mathbf{X}_t = \sum_{i=1}^{d} f_i(\mathbf{Y}_t)\,d\mathbf{X}_t^i, \quad t \in [0,T], \\
\mathbf{Y}_0 &= \xi,
\end{array}
\right.
\label{eq:rde_lifted}
\end{equation}
where $\mathbf{X}$ is the rough path lift of $X$, and $\mathbf{Y}$ is an $\mathbf{X}$-controlled path. The solution theory for~\eqref{eq:rde_lifted} crucially depends on algebraic properties such as the shuffle product and analytic bounds on the rough path lift.

The classical rough path theory is founded on the shuffle Hopf algebra $\mathcal{H}_\shuffle$ of words, which encodes iterated integrals along smooth paths. By replacing this with richer combinatorial Hopf algebras, one obtains generalized rough path frameworks adapted to more complex algebraic and analytic structures. For example, using the Butcher-Connes-Kreimer (BCK) Hopf algebra $\mathcal{H}_{{\rm BCK}}$
of rooted forests yields branched rough paths~\mcite{CK98,Gub1}. A further refinement leads to planarly branched rough paths via the Munthe-Kaas-Wright (MKW) Hopf algebra $\mathcal{H}_{{\rm MKW}}$ of planar rooted forests~\mcite{MKW}, capturing planar composition structures relevant in numerical settings.
More recently, LOT rough paths have been developed from the LOT Hopf algebra $\mathcal{H}_{{\rm LOT}}$
of multi-indices, introduced by Linares, Otto and Tempelmayr (LOT)~\mcite{BH,LOT,ZGM}. This construction provides a natural connection to regularity structures via a top-down approach rooted in multi-index analysis~\cite{BL24,LOTT}.
An overview of these rough path frameworks, along with their interconnections, is provided in~\mcite{Man25}, and summarized in the diagram below.
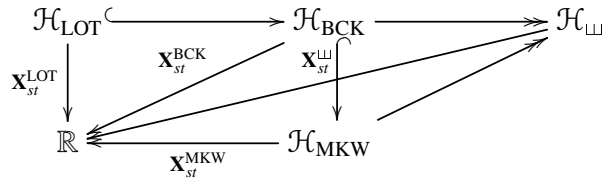
\begin{displaymath}
\xymatrix@C=2.2cm@R=1cm{
\mathcal{H}_{{\rm LOT}} \ar@{^(->}[r] \ar[d]_{{\bf X}_{st}^{{\rm LOT}}}& \mathcal{H}_{{\rm BCK}}  \ar@<0.1cm>@{^(->}[d]\ar@{->>}[r] \ar[ld]_{{\bf X}_{st}^{{\rm BCK}}}& \mathcal{H}_\shuffle\ar[dll]_{{\bf X}_{st}^{\shuffle}}\\
\mathbb{R}&\mathcal{H}_{{\rm MKW}}\ar@{->>}[ur]\ar[l]^{{\bf X}_{st}^{{\rm MKW}}}&
}
\end{displaymath}

Another fundamental notion accompanying the theory of rough paths is that of a controlled rough path, a concept that captures how a path can be locally expressed in terms of another, typically more irregular, reference path. The shuffle algebraic formulation of controlled rough paths was originally formulated by Gubinelli in~\cite{Gu04}. Specifically, given a rough path $\mathbf{X}$ taking values in the shuffle Hopf algebra $\mathcal{H}_\shuffle$, an $\mathbf{X}$-controlled rough path is a map $\mathbf{Y} : [0,T] \to \mathcal{H}_\shuffle$ satisfying the following relation:
\[
\langle \tau, \mathbf{Y}_t \rangle = \langle \mathbf{X}_{st} \otimes \tau, \mathbf{Y}_s \rangle + \text{small remainder},
\]
for all basis elements $\tau$ of the shuffle algebra. This identity expresses how the increment $\mathbf{Y}_{st} := \mathbf{Y}_t - \mathbf{Y}_s$ can be approximated in terms of the rough driver $\mathbf{X}_{st}$, with an error term that is small in a suitable analytic sense.
This framework was later extended beyond the shuffle setting. A branched version of controlled rough paths was developed in~\cite{Gub1} and further explored in~\cite{Kel}. More recently, a planarly branched version was introduced in~\cite{GLM24}. The most general formulation, applicable in an abstract Hopf algebraic setting, was proposed in~\cite{ZGLM25}, encompassing all previous cases and offering a unified algebraic framework for modeling controlled rough paths across various combinatorial objects.

\subsection{Comparison and outline}
The present work differs from Hairer's seminal contribution~\cite{Hair11} both in scope and in methodology. In~\cite{Hair11}, the rough path approach is developed under the assumption that the spatial regularity of the solution lies in the regime $\alpha \in(\frac{1}{3}, \frac{1}{2})$, which ensures that the nonlinear term $g(u)\,\partial_x u$ can be interpreted via a rough integral with relatively mild analytic constraints. By contrast, the analysis in this paper systematically targets the entire subcritical range $\alpha \in(0, \frac{1}{2})$, where the spatial roughness is substantially more severe and classical rough path estimates no longer suffice.

From a methodological standpoint, our approach introduces new quantitative bounds for controlled rough paths under composition with regular functions and develops a refined scaling argument for rough integrals involving heat kernels. These estimates go beyond those available in~\cite{Hair11} and are specifically designed to capture stochastic cancellation effects that become dominant at low regularity. While Hairer's framework relies on a fixed rough path lift of the linear solution and corresponding analytic bounds, our work---such as Propositions~\ref{pro:regu}, \ref{thm:inte2}, \ref{prop:sca1} and Theorem~\ref{thm:eu1}---extends this strategy by strengthening the underlying rough integral estimates, thereby enabling a robust pathwise well-posedness theory in a strictly rougher regime . In this sense, the present paper can be viewed as a genuine extension of~\cite{Hair11}, both analytically and conceptually, rather than a direct adaptation.

The structure of the paper is as follows.
In Section~\ref{ss:sec2}, we first review the fundamental notions of rough paths and controlled rough paths, along with their composition with sufficiently regular functions. We then establish an upper bound for the composition of a controlled rough path with a regular function, valid for any roughness parameter $\al \in (0, \frac{1}{2})$ (Proposition~\ref{pro:regu}). These topics form the analytical backbone of the present work.

Section~\ref{ss:sec3} is devoted to the theory of rough integrals. We recall the concept of rough integral and establish a crucial bound related to the rough integral (Proposition~\ref{thm:inte2}). Furthermore, we present an important result concerning the behavior of scaled functions under rough integration, formalized in Proposition~\ref{prop:sca1}.

In Section~\ref{ss:sec4}, we develop the pathwise theory of existence and uniqueness for Burgers-type SPDEs. Specifically, we establish that both local and global mild solutions $u$ to~\eqref{eq:bspde} exist and are unique when the spatial regularity parameter
$\beta$ lies in the range $\tred{(0, \frac{1}{2})}$, as detailed in Theorems~\ref{thm:eu1} and~\ref{thm:eu2}.
The result of Hairer~\cite[Theorem 3.6]{Hair11} serves as a corollary for us here.
This provides a foundational result for the well-posedness of the considered SPDEs in the low-regularity setting relevant to rough path analysis.

 \vskip 0.1in
{\bf Notation.}
Throughout this paper, we work over the field $\mathbb{R}$ of real numbers, which serves as the base field for all vector spaces, tensor products, algebras, coalgebras, and linear maps under consideration.
We fix two positive integers: $d$, representing the dimension of the ambient space in which rough paths take values, and $m$, denoting the dimension associated with controlled rough paths.
Let $\alpha \in (0, 1]$, and let $V$ be a Banach space. For a continuous path
\[
X: [0,T] \to V, \quad t \mapsto X_t := X(t),
\]
we define the increment of $X$ over an interval $[s,t]$ by
$
\delta X_{s,t} := X_t - X_s.
$
We denote by $\mathcal{C}^\alpha([0,T], V)$ the space of $V$-valued continuous paths on $[0,1]$ equipped with the norm
\begin{equation}
\|X\|_{\mathcal{C^\alpha}} := \|X\|_{\alpha} + \|X\|_{\infty},
\mlabel{eq:note1}
\end{equation}
where
\begin{equation}
\|X\|_{\alpha} := \sup_{s \neq t \in [0,T]} \frac{\|X_t - X_s\|_V}{|t - s|^{\alpha}}, \qquad
\|X\|_{\infty} := \sup_{t \in [0,T]} \|X_t\|_V.
\mlabel{eq:note2}
\end{equation}
This space captures paths that are both $\alpha$-H\"older continuous and bounded, which is the standard setting for rough path analysis.
%

 \vskip 0.1in
\section{An upper bound of the composition with regular functions}\label{ss:sec2}
In this section, we first recall the concepts of rough path, controlled rough path and composition of controlled rough paths with regular functions, which are the primary subjects examined in the present article. Then we obtain an upper bound of the composition of controlled rough paths with regular functions, with any roughness $\al \in(0,1]$.

\subsection{Rough paths and controlled rough paths}
Let $d\in \ZZ_{\geq 1}$ and $T^{k}(\V) := (\V)^{\otimes k}$ be the $k$-th tensor power of $\V$ for any $k\in \ZZ_{\geq 0}$, with the convention that $T^{0}(\V) := \RR$. Construct the direct sum
$$T(\V):= \bigoplus_{k\geq 0} T^{k}(\V).$$
Elements $\omega\in T^{k}(\V)$ are said to have degree $|\omega|:= k$. The space $T(\V)$ equipped respectively with the tensor product $\otimes$ and the shuffle product $\shuffle$ can be turned into the tensor algebra $(T(\V), \otimes, \etree)$ and shuffle algebra $(T(\V), \shuffle, \etree)$, where $\etree:\RR\rightarrow T(\V)$ is the unit given by $\etree(1)=1$.
Since $\V$ is of finite dimension, we can identify $T^{k}(\V)$ with its dual space for each $k\in \ZZ_{\geq 0}$.
Further, one has the (connected and graded) shuffle Hopf algebra
\begin{align*}
T(\V):= (T(\V), \shuffle, \etree, \Delta_{\otimes},\etree^*),
\end{align*}
where the coalgebra $(T(\V), \Delta_{\otimes},\etree^*)$ is obtained by taking the graded dual, equal to the finite dual in this case, of the tensor algebra $(T(\V), \otimes, \etree)$.
The graded dual $T(\V)^{g}$ of the shuffle Hopf algebra $T(\V)$ is the (connected and graded) tensor Hopf algebra
\begin{align*}
T(\V)^{g}:= (T(\V),\otimes, \etree, \Delta_{\shuffle},\etree^*),
\end{align*}
where the coalgebra $(T(\V), \Delta_{\shuffle},\etree^*)$ is from the graded dual of the shuffle algebra $(T(\V), \shuffle, \etree)$.
Here we employ the natural pairing
$$\langle\, , \rangle: T(\V)^{g} \otimes T(\V) \rightarrow \RR, \quad \omega_1\otimes \omega_2\mapsto \langle \omega_1, \omega_2\rangle:= \omega_1(\omega_2).$$

Now we consider the truncation of $T(\V)$ and $T(\V)^{g}$.
For each $N\in \ZZ_{\geq 0}$, the spaces
$$T^{\le N}(\RR^d) := \bigoplus_{0\leq k\leq N} T^k(\RR^d)=: T^{\le N}(\RR^d)^g$$
are endowed with the structure of a connected, graded, and finite-dimensional algebra and coalgebra, as described below.
On the one hand, the vector subspace
$$I_N:= \bigoplus_{k>N} T^k(\V)\leq T(\V)$$ is a graded ideal (but not a bi-ideal), hence $(T(\V), \sha, 1)/I_N$ is a graded algebra.
On the other hand, the restriction of the projection $\pi_N:T(\V)\to\hskip -9pt\to T(\V)/I_N$ to the subcoalgebra $T^{\le N}(\V)$ is an isomorphism of graded vector spaces.
The graded algebra structure of $T(\V)/I_N$ can therefore be transported to $T^{\le N}(\V)$, making  it both an algebra and a coalgebra, denoted by
$$\tn:= (T^{\le N}(\V), \shuffle, \etree, \Delta_{\otimes},\etree^*)$$
by a slight abuse of notations. Since $T^{\le N}(\V)$ is of finite dimension,
we also have the (graded) dual algebra/coalgebra
$$\tng = (T^{\le N}(\V), \otimes, \etree^*, \Delta_{\shuffle},\etree)$$ under the above pairing.

The following is the concept of a (weakly geometric) rough path.

\begin{definition}\cite{Ly98}
Let $\alpha \in(0,1]$ and $N= \lfloor \frac{1}{\alpha} \rfloor$. An $\alpha$-H\"{o}lder rough path is a map
$$\X= (X^0, X^1, \ldots, X^N): [0,T]^2\to \tng$$
such that
\begin{enumerate}
\item  (Chen's relation) \quad $\X_{s,u} \pro \X_{u,t} =\X_{s,t}$, \quad $\forall s, u, t\in \left[ 0, T\right]$.
\\
\item  $\sup\limits_{s\ne t\in [ 0,T ] } \frac{| \langle \X_{s,t},\, \omega  \rangle  |}{| t-s | ^{k\alpha  } } <\infty$,\quad $\forall s, t\in \left[ 0, T\right]$ and $\omega\in T^{k}(\V) $ with $0\leq k\leq N$.
\end{enumerate}
Further, it is called {\bf weakly geometric} if
\begin{equation}
\text{(Shuffle relation)} \quad \langle \X_{s,t}, \omega_1 \shuffle \omega_2\rangle  = \langle \X_{s,t}, \omega_1\rangle\langle \X_{s,t}, \omega_2\rangle, \quad \forall \omega_1,  \omega_2\in T^{\le N}(\V),
\mlabel{eq:shuffle}
\end{equation}
and called {\bf above a path}
$X:[0,T]\rightarrow \mathbb{R}^d$ if
\begin{equation*}
\langle \X_{s,t} , \omega \rangle= \langle X_{t}, \omega\rangle - \langle X_{s}, \omega\rangle, \quad \forall \omega\in T^{1}(\V).
\end{equation*}
In this case, we also call $X$ is lifted to $\X$.
\mlabel{def:wrp}
\end{definition}

We write $\RRP$ for the set of $\alpha$-H\"{o}lder rough paths, and $\RP$ for the set of $\alpha$-H\"{o}lder
weakly geometric rough paths.
For each $\X\in \RRP$, the zeroth component $X^{0}:[0,T]^2\rightarrow \RR$ must be the constant map with value $1$.
So for the sequence of the paper, we always write $X^0$ as 1.
For $\X\in \RRP$, define
\begin{equation}
\|\X\|_{\alpha}:=\sum_{i=1}^{N}\|X^i\|_{i\alpha}.
\mlabel{eq:note4}
\end{equation}

Here is the concept of a controlled rough path. \vspace{0.1in}

\begin{definition}~\cite{FH20,Gu04}
Let $\alpha \in(0,1]$, $N= \lfloor \frac{1}{\alpha} \rfloor$ and $\X\in \RRP $. A path $ \Y: [0, T]\to T^{\le N-1}(\RR^d)$ is called an $\X$-controlled rough path if
\begin{equation*}
\|R\Y^\omega\|_{(N-|\omega|)\alpha} < \infty,
\end{equation*}
where
\begin{equation*}
R\Y^{\omega}_{s,t}:= \langle \omega, \Y_t  \rangle- \langle \X_{s,t}\otimes \omega, \Y_s  \rangle, \quad \forall \omega\in T^{k}(\V)\,\text{ with }\, 0\leq k\leq N-1.
\end{equation*}
Further we call  $\Y$ above a path $Y:[0,T]\rightarrow \mathbb{R}$ if
$\langle \etree, \Y_{t} \rangle=Y_{t}.$
\mlabel{def:crp}
\end{definition}

The above concept in detail can be recast as follows.

\begin{remark} Let $Y:[0, T] \to \RR^m$ be a path.
 Let $\al \in(0, 1]$ and $\X= (1, X^1, \ldots, X^N)\in \RRP$. The path
$$\Y = (Y^0, \ldots, Y^{N-1}):[0, T]\to \Big(\RR^m, \mathcal{L}(\RR^d, \RR^m),\ldots, \mathcal{L}((\RR^d)^{\otimes (N-1)}, \RR^m)   \Big)$$

is an $\X$-controlled rough path above $Y$ if and only if~\cite[Definition 2.2]{BG22}
\[
\|R^i\|_{(N-i)\al}<\infty\quad \text{and} \quad Y^0=Y,
\]
where
\begin{equation}
R^i_{s, t}:=
\begin{cases}
 Y_t^i-Y_s^i-\sum_{j=1}^{N-1-i}Y^{i+j}_sX^j_{s, t},\quad i=0, \ldots, N-2;\\
Y_t^{N-1}-Y_s^{N-1},\hspace{2.42cm} i=N-1.
\end{cases}
\mlabel{eq:crp}
\end{equation}

\mlabel{rem:crp}
\end{remark}

For $\X\in \D$, the set of $\X$-controlled rough paths $\Y = (Y^0, \ldots, Y^{N-1})$ given in Remark~\ref{rem:crp} is a Banach space~\cite{BG22} under the norm
\begin{equation}
\|\Y\|_{\X,\al}:=\sum_{i=0}^{N-1}|Y_0^i|+\sum_{i=0}^{N-1}\|R^i\|_{(N-i)\al},
\mlabel{eq:norm}
\end{equation}
denoted by $\CC([0, T], \RR^m)$.

\begin{remark}
For a rough path $\X= (1, X^1, \ldots, X^N)$ above $X$, the path $X:[0, T] \to \RR^d$ can be identified to an $\X$-controlled rough path $(X,\id,0,\ldots, 0)$, as
\begin{equation*}
R^i_{s, t}:=
\begin{cases}
Y_t^0-Y_s^0-\sum_{j=1}^{N-2}Y^{1+j}_sX^j_{s, t}=X_t-X_s-\id_sX^1_{s,t}=0,\quad i=0;\\
 Y_t^i-Y_s^i-\sum_{j=1}^{N-1-i}Y^{i+j}_sX^j_{s, t}=\id_t-\id_s=0,\hspace{1.82cm} i=1, \ldots, N-2;\\
Y_t^{N-1}-Y_s^{N-1}=0,\hspace{5.68cm} i=N-1.
\end{cases}
\end{equation*}
Here $\id:[0, T] \to \mathcal{L}(\RR^d, \RR^d)$ is the constant path with value the identity map $\id$ for all times.
\end{remark}

\subsection{Composition with regular functions}
In this subsection, we first review that the composition of a controlled rough path with a regular function is still a controlled rough path. Based on this, we then give an upper bound on the norm of the newly obtained controlled rough path.
For $k\in \ZZ_{\geq 1}$, denote by
\begin{align*}
\mathcal{C}^k_b(\RR^m, \RR^n):= \{ \f:\RR^m \to \RR^n \text{ is } k \text{ times continuously differentiable
  and } \|D^jf\|_{\infty} < \infty, j=0, \ldots, k\}.
\end{align*}
Here $$D^j\varphi:\mathbb R^m\to\mathcal L\big((\RR^m)^{\otimes j},\mathbb R^n\big)$$ denotes the $j$-th differential of $\varphi$.
For each $j\in \ZZ_{\geq 1}$ and the truncated algebra $$T^{\le N}(\V)=(T(\V), \otimes , 1) /I_N,$$
the $\Big(T^{\le N}(\V)\Big)^{\otimes j}$ is also an algebra with respect to the component-wise multiplication, and there is an algebra homomorphism~\mcite{BG22}
\begin{equation*}
\delta_j: T^{\le N}(\V)\to  \Big(T^{\le N}(\V)\Big)^{\otimes j},
\end{equation*}
induced by
\begin{equation}
\begin{aligned}
\delta_j(v):=&\ v\otimes 1 \otimes \cdots \otimes 1+\dots+ 1 \otimes \cdots \otimes 1 \otimes v,  \quad \forall v\in \RR^d.
\end{aligned}
\mlabel{eq:map}
\end{equation}

The following is the concept of the composition of a controlled rough path with a regular function.
\begin{definition}~\cite[(4.2)]{BG22}
Let $\f\in \mathcal{C}^N_b(\RR^m, \RR^n)$, $\X\in \RP$ and $\Y = (Y^0, \ldots, Y^{N-1})\in \CC([0, T], \RR^m)$ above $Y$. Define a controlled rough path
$\f(\Y) = (\f(Y)^0, \ldots, \f(Y)^{N-1})\in \CC([0, T], \RR^n)$ via $\f(Y)^0_t:=\f(Y^0_t)$ and
\begin{equation}
\f(Y)^r_t:=\sum_{j=1}^{r}{1\over j!} D^j\f(Y^0_t)\Big(\sum_{i_1+\dots+i_j=r}(Y_t^{i_1}\otimes\dots \otimes Y_t^{i_j}) \circ \delta_j\mid_{(\RR^d)^{\otimes r}} \Big)\in \mathcal{L}((\RR^d)^{\otimes r}, \RR^m)
\mlabel{eq:regu}
\end{equation}
for $r=1,\ldots, N-1$ and $1\le i_1,\ldots, i_j\le N-1$. Here
\[
(\RR^d)^{\otimes r} \overset{\delta_j}{\rightarrow} \bigoplus_{i_1+\dots+i_j=r}(\RR^{d})^{\otimes i_1}\otimes \dots \otimes (\RR^{d})^{\otimes i_j} \xlongrightarrow{Y_t^{i_1}\otimes\dots \otimes Y_t^{i_j}} (\RR^m)^{\otimes j}
\xlongrightarrow{D^j\f(Y^0_t)} \RR^n.
\]
\mlabel{def:regu}
\end{definition}

Next, we establish an upper bound for the controlled rough path obtained above.

\begin{proposition}
Let $\al \in(0,1]$, $\X\in \RP$, $\Y = (Y^0, \ldots, Y^{N-1})\in \CC([0, T], \RR^m)$ and $\f(\Y) = (\f(Y)^0, \ldots, \f(Y)^{N-1})\in \CC([0, T], \RR^n)$ be given by (\ref{eq:regu}). Then
\begin{equation*}
\|\f(\Y) \|_{\X,\al}\le C_{\al, T}\Big(\sum_{i=0}^N\|D^i\f\|_{\infty}\Big)\Big(\sum_{i=1}^{N-1}\|Y^i\|_{\C}^l\Big)(1+\|\Y\|_{\X,\al})^k,
\end{equation*}
where $C_{\al, T}\in \RR$ and $l, k \in \RR_{>0}$.
\mlabel{pro:regu}
\end{proposition}

\begin{proof}
By~\cite[Theorem 4.1]{BG22}, we have
\begin{equation}
\|\f(\Y) \|_{\X,\al}\le C\Big(\sum_{i=0}^N\|D^i\f\|_{\infty}\Big)T^q\Big(\max_{1\le i\le N-1}|Y_0^i|\Big)^l\|\X \|_{\al}^r\|\Y\|_{\X,\al}^k \quad \text{for some $q, l, r, k \in \RR_{>0}$}.
\mlabel{eq:reg1}
\end{equation}
In order to reach the desired conclusion, we need to deal with $\Big(\max_{1\le i\le N-1}|Y_0^i|\Big)^l$ in (\ref{eq:reg1}). We have the following estimate
\begin{align}
\Big(\max_{1\le i\le N-1}|Y_0^i|\Big)^l\le&\ \Big(\max_{1\le i\le N-1}(|Y_t^i|+|Y_t^i-Y_0^i|)\Big)^l\nonumber\\
\le&\ \Big(\max_{1\le i\le N-1}(\|Y^i\|_{\infty}+T^\al\|Y^i\|_{\al})\Big)^l\hspace{2cm} (\text{by (\ref{eq:note2})})\nonumber\\
\le&\ \Big(\max_{1\le i\le N-1}(1+T^\al)(\|Y^i\|_{\infty}+\|Y^i\|_{\al})\Big)^l\nonumber\\
=&\ \Big((1+T^\al)\max_{1\le i\le N-1}\|Y^i\|_{\C}\Big)^l\hspace{2cm} (\text{by (\ref{eq:note1})})\nonumber\\
=&\ (1+T^\al)^l\Big(\max_{1\le i\le N-1}\|Y^i\|_{\C}\Big)^l\nonumber\\
\le&\ (1+T^\al)^l\sum_{i=1}^{N-1}\|Y^i\|_{\C}^l. \mlabel{eq:reg2}
\end{align}
Substituting (\ref{eq:reg2}) into (\ref{eq:reg1}), we obtain
\begin{align*}
\|\f(\Y) \|_{\X,\al}\le&\ C\Big(\sum_{i=0}^N\|D^i\f\|_{\infty}\Big)T^q(1+T^\al)^l\Big(\sum_{i=1}^{N-1}\|Y^i\|_{\C}^l\Big)\|\X \|_{\al}^r\|\Y\|_{\X,\al}^k\\
=&\ \Big(CT^q(1+T^\al)^l\|\X \|_{\al}^r\Big)\Big(\sum_{i=0}^N\|D^i\f\|_{\infty}\Big)\Big(\sum_{i=1}^{N-1}\|Y^i\|_{\C}^l\Big)\|\Y\|_{\X,\al}^k\\
\le&\ \Big(CT^q(1+T^\al)^l\|\X \|_{\al}^r\Big)\Big(\sum_{i=0}^N\|D^i\f\|_{\infty}\Big)\Big(\sum_{i=1}^{N-1}\|Y^i\|_{\C}^l\Big)(1+\|\Y\|_{\X,\al})^k\\
\le&\ C_{\al, T}\Big(\sum_{i=0}^N\|D^i\f\|_{\infty}\Big)\Big(\sum_{i=1}^{N-1}\|Y^i\|_{\C}^l\Big)(1+\|\Y\|_{\X,\al})^k,
\end{align*}
where the constant $C_{\al, T}:=C_{\al, T, \X}$ is dependent on $\al, T, \X$.
This completes the proof.
\end{proof}

We conclude this section with the following observation, which will be used later.

\begin{remark}
For a function $g\in \mathcal{C}^{N\al}([0, T],\RR)$ and a controlled rough path $\Y=(Y^0, \ldots, Y^{N-1})\in \CC([0, T], \RR^m)$, if we take $\f(Y):=gY$ in (\ref{eq:regu}), then we obtain a new controlled rough path
$$\f(\Y):=(gY^0, \ldots, gY^{N-1})\in \CC([0, T], \RR^m).$$
Furthermore, there exist a constant $C_{\alpha, T}\in \RR$ such that
\begin{align}
\|\f(\Y)\|_{\X,\al}
= &\ \sum_{i=0}^{N-1}|g_0Y_0^i|+\sum_{i=0}^{N-1}\|R^{gY, i}\|_{(N-i)\al} \hspace{1cm} (\text{by (\ref{eq:norm})})\nonumber\\
\le & \sum_{i=0}^{N-1}\|g\|_{\C}|Y_0^i|+\sum_{i=0}^{N-1}\|g\|_{(N-i)\al}\|R^{Y, i}\|_{(N-i)\al}\hspace{1cm} (\text{by (\ref{eq:note2})}) \nonumber\\
\le & \sum_{i=0}^{N-1}\|g\|_{\C}|Y_0^i|+\sum_{i=0}^{N-1}\|g\|_{\mathcal{C}^{(N-i)\al}}\|R^{Y, i}\|_{(N-i)\al}\nonumber\\
\le & (1+T^\al+\dots+T^{(N-1)\al})\|g\|_{\mathcal{C}^{N\al}}\Big(\sum_{i=0}^{N-1}|Y_0^i|+\sum_{i=0}^{N-1}\|R^{Y, i}\|_{(N-i)\al}\Big)\nonumber\\
= &\ (1+T^\al+\dots+T^{(N-1)\al})\|g\|_{\mathcal{C}^{N\al}}\|\Y\|_{\X,\al}\hspace{2cm} (\text{by (\ref{eq:norm})})\nonumber\\
=&: C_{\al, T}\|g \|_{\mathcal{C}^{N\al}}\|\Y\|_{\X,\al}. \mlabel{eq:regu1}
\end{align}
\end{remark}

\section{Rough integrals and scaled functions}\label{ss:sec3}
In this section, we begin by reviewing the concept of the rough integral and providing two estimates for its bounds. We then conclude with an important result concerning scaled functions.

\subsection{Rough integrals}
To set the stage, we briefly recall the concept of the rough integral and highlight some key properties needed for the upcoming analysis.

\begin{definition}\cite{BG22,Gu04}
Let $\al \in(0, 1]$, $\X = (1, X^1, \ldots, X^N)\in \D$ above $X$  and $\Z = (Z^0, \ldots, Z^{N-1})\in \CC([0, T], \mathcal{L}(\RR^d, \RR^m))$ above $Z$. Define the rough integral of $Z$ against $\X$ by
\begin{equation*}
\int_0^1 Z_r\,dX_r:=\lim_{|\pi | \to 0} \sum_{[s, t ]\in \pi}\sum_{i=0}^{N-1}Z_s^iX_{s, t}^{i+1} \in \RR^m,
\end{equation*}
where $\pi$ is an arbitrary partition of $[0, T]$. Notice that, for $i=0,\ldots, N-1$,
$$Z_t^i\in \mathcal{L}\Big((\RR^d)^{\otimes i}, \mathcal{L}(\RR^d, \RR^m)\Big)\cong \mathcal{L}\Big((\RR^d)^{\otimes (i+1)}, \RR^m\Big).$$
\mlabel{def:int}
\end{definition}
\vspace{-0.6cm}
A lemma will never hurt.

\vspace{0.3cm}

\begin{lemma}~\cite{Gu04}
Under the settting in~Definition~\mref{def:int}, we have the following estimation
\begin{equation}
\Big|\int_s^t Z_r \, dX_r-\sum_{i=0}^{N-1}Z_s^iX_{s, t}^{i+1} \Big|\le C_{\al}\|\X\|_{\al}\|\Z\|_{\X,\al}|t-s|^{(N+1)\al},
\mlabel{eq:inte1}
\end{equation}
where $C_{\al}\in \RR$.
\end{lemma}

We now state a result that bounds the norm of a path from the above rough integral.

\begin{proposition}
With the settting in~Definition~\mref{def:int},
\begin{equation*}
\Big\|\int_0^{\bullet}  Z_{0,r} \, dX_r \Big\|_{\al}\le C_{\al, T}\|\X\|_{\al}\|\Z\|_{\X,\al} ,
\end{equation*}
where $C_{\al,T}\in \RR$.
\mlabel{thm:inte2}
\end{proposition}

\begin{proof}
To prove this inequality, we have
\begin{align}
&\ \Big|\int_s^t \delta Z_{0,r} \, dX_r\Big|\nonumber\\
\le&\ \Big|\int_s^t \delta Z_{0,r} \, dX_r-\sum_{i=0}^{N-1}\delta Z_{0,s}^iX_{s, t}^{i+1}\Big|+\Big| \sum_{i=0}^{N-1}\delta Z_{0,s}^iX_{s, t}^{i+1}\Big|  \nonumber\\
\le&\  C_{\al}\|\X\|_{\al}\|\Z\|_{\X,\al}|t-s|^{(N+1)\al}+\sum_{i=0}^{N-1}|\delta Z_{0,s}^i|\,|X_{s, t}^{i+1}|\hspace{1.5cm} (\text{by (\ref{eq:inte1})})\nonumber\\
\le&\ C_{\al}\|\X\|_{\al}\|\Z\|_{\X,\al}|t-s|^{(N+1)\al}+2\sum_{i=0}^{N-1}\|Z^i\|_{\infty}|X_{s, t}^{i+1}|.\mlabel{eq:inte2}
\end{align}
This implies
\begin{align*}
&\ \Big\|\int_0^{\bullet} \delta Z_{0,r}  \, dX_r \Big\|_{\al}\\
=&\ \sup_{s\ne t\in [ 0,T ] } \dfrac{\Big|\int_s^t \delta Z_{0,r}  \, dX_r\Big|}{|t-s|^{\al}}\\
\le&\ \sup_{s\ne t\in [ 0,T ] }\frac{C_{\al}\|\X\|_{\al}\|\Z\|_{\X,\al}|t-s|^{(N+1)\al}}{|t-s|^{\al}}+2\sum_{i=0}^{N-1}\|Z^i\|_{\infty}\sup_{s\ne t\in [ 0,T ] }\frac{|X_{s, t}^{i+1}|}{|t-s|^{\al}} \hspace{2cm} (\text{by (\ref{eq:inte2})})\\
\le&\ C_{\al}T^{N\al}\|\X\|_{\al}\|\Z\|_{\X,\al}+2\sum_{i=0}^{N-1}\|Z^i\|_{\infty}T^{i\al}\|X^{i+1}\|_{(i+1)\al}\\
\le&\ C_{\al}T^{N\al}\|\X\|_{\al}\|\Z\|_{\X,\al}+2\sum_{i=0}^{N-1}C\|\Z\|_{\X,\al}T^{i\al}\|X^{i+1}\|_{(i+1)\al}\\
\le&\ C_{\al}T^{N\al}\|\X\|_{\al}\|\Z\|_{\X,\al}+2C\sum_{i=0}^{N-1}T^{i\al}\|\X\|_{\al}\|\Z\|_{\X,\al}\hspace{2cm} (\text{by (\ref{eq:note4})})\\
=&\ C_{\al, T}\|\X\|_{\al}\|\Z\|_{\X,\al} ,
\end{align*}
where $$C_{\al, T}:=C_{\al}T^{N\al}+2C\sum_{i=0}^{N-1}T^{i\al}.$$
This completes the proof.
\end{proof}

\subsection{Integration with scaled functions}
To address the nonlinear equation using the heat semigroup, we estimate a new rough integral involving an integrand multiplied by a smooth function. To this end, we recall a family of scaled functions from~\cite{Hair11}
$$ \mathcal{C}^1_1(\RR, \RR):=\{f:\RR \to \RR \mid f\, \text{ has continuous first derivative and }\,   \|f\|_{1, 1} < \infty\},$$
where
\begin{equation}
\|f\|_{1, 1}:=\sum_{n\in \mathbb{Z} }\sup_{0\le t\le 1}(|f(n+t)|+|f'(n+t)|).
\mlabel{eq:scal0}
\end{equation}


\begin{proposition}
Let $\al \in(0, 1]$, $\X = (1, X^1, \ldots, X^N)\in \D$ above $X$.
Suppose $\Z = (Z^0, \ldots, Z^{N-1})\in \CC([0, T], \mathcal{L}(\RR^d, \RR^m))$ above $Z$ and $f\in \mathcal{C}^1_1(\RR, \RR)$. Then
for each $\lambda \in \RR_{\geq 1}$,
the $\X$-controlled rough path
$$(f(\lambda \cdot) Z_{\cdot}^0, \ldots, f(\lambda \cdot) Z^{N-1}_{\cdot}) \in \CC([0, T], \mathcal{L}(\RR^d, \RR^m))$$
given by Definition~\mref{def:regu} has the estimation
\begin{equation}
\Big|\int_0^1 f(\lambda t)Z(t)\,dX(t)\Big|\le C_{\alpha, T}\lambda^{-\al}\|f\|_{1, 1}\|\Z\|_{\X,\al}\|\X\|_{\al},
\mlabel{eq:sca1}
\end{equation}
where $C_{\al, T}\in \RR$.
\mlabel{prop:sca1}
\end{proposition}

\begin{proof}
Without loss of generality, we may assume that $\lambda\in \ZZ_{\geq 1}$.
For an integer $0\leq k\le T(\lambda-1)$, we reparameterize the rough path
$\X$ to obtain a new rough path
$$\X_{\lambda, k}\in \mathcal{D}^{\alpha}\Big([-k, T\lambda-k], \RR^d\Big)$$
be setting
\begin{align*}
\X_{\lambda, k}(s, t):=&\ \Big(X^1_{\lambda, k}(s, t), \ldots, X^N_{\lambda, k}(s, t)\Big)\\
:=&\ \Big(X^1(\frac{s+k}{\lambda}, \frac{t+k}{\lambda}), \ldots, X^N(\frac{s+k}{\lambda}, \frac{t+k}{\lambda})\Big).
\end{align*}
Since $0\leq k \le T(\lambda-1)$, we have $-k\leq 0\leq T\leq T\lambda -k$ and so the rough path $\X_{\lambda, k}$ can be restricted to the interval $[0,T]$,
resulting in a new rough path
$$\X_{\lambda, k}\in \mathcal{D}^{\alpha}\Big([0, T], \RR^d\Big),$$
which we still denote as $\X_{\lambda, k}$.
Further, the path $\Z_{\lambda, k}$ given by
$$\Z_{\lambda, k}(t):=\Big(Z^0_{\lambda, k}(t), \ldots, Z^{N-1}_{\lambda, k}(t)\Big):=\Big(Z^0(\frac{t+k}{\lambda}), \ldots, Z^{N-1}(\frac{t+k}{\lambda})\Big)$$
is an $\X_{\lambda, k}$-controlled rough path in $\mathcal{C}_{\X_{\lambda, k}}^{\alpha}\Big([0, T], \mathcal{L}(\RR^d, \RR^m)\Big)$.

Now to prove~(\mref{eq:sca1}), we first have
\begin{equation}
\begin{aligned}
\|X^1_{\lambda, k}\|_{\al}=&\ \sup_{s\ne t\in [0, T] }\dfrac{|X^1_{\lambda, k}(s, t)|}{|t-s|^{\al}}
=\lambda^{-\al}\sup_{s\ne t\in [0, T] }\dfrac{|X^1(\frac{s+k}{\lambda}, \frac{t+k}{\lambda})|}{|\frac{t+k}{\lambda}-\frac{s+k}{\lambda}|^{\al}} \\
\le&\ \lambda^{-\al}\sup_{s\ne t\in [-k, T\lambda-k] }\dfrac{|X^1(\frac{s+k}{\lambda}, \frac{t+k}{\lambda})|}{|\frac{t+k}{\lambda}-\frac{s+k}{\lambda}|^{\al}} =\lambda^{-\al}\|X^1\|_{\al}.
\mlabel{eq:sca3}
\end{aligned}
\end{equation}
Similarly,
\begin{equation}
\|X^i_{\lambda, k}\|_{i\al}\le \lambda^{-i\al}\|X^i\|_{i\al}, \quad \text{for $i=2,\ldots, N$}.
\mlabel{eq:sca4}
\end{equation}
For the rough integral of the left-hand side of~(\mref{eq:sca1}),
\begin{align}
\sum_{k=0}^{\lambda-1}\int_0^1f(t+k)Z_{\lambda, k}(t)\,dX_{\lambda, k}(t)=&\ \sum_{k=0}^{\lambda-1}\int_0^1f(t+k)Z(\frac{t+k}{\lambda})\,dX(\frac{t+k}{\lambda})\nonumber\\
=&\ \sum_{k=0}^{\lambda-1}\int_{\frac{k}{\lambda}}^{\frac{1+k}{\lambda}}f(\lambda u)Z(u)\,dX(u)\hspace{1cm} (\text{by setting $u:=\frac{t+k}{\lambda}$})\nonumber\\
=&\ \int_0^1f(\lambda u)Z(u)\,dX(u)\nonumber\\
=&\ \int_0^1f(\lambda t)Z(t)\,dX(t).\mlabel{eq:scal2}
\end{align}
Setting $f_k(t):=f(t+k)$ with $t\in [0, 1]$,
\begin{align}
\sum_{k=0}^{\lambda-1}\|f_k\|_{\mathcal{C}^{N\al}}=&\ \sum_{k=0}^{\lambda-1}(\|f_k\|_{N\al}+\|f_k\|_{\infty}) \hspace{2cm} (\text{by (\ref{eq:note2})})\nonumber\\
=&\ \sum_{k=0}^{\lambda-1}\Big(\sup_{s\ne t\in [0, 1] }\dfrac{|f_k(t)-f_k(s)|}{|t-s|^{N\al}}+\sup_{ t\in [0, 1] }|f_k(t)|  \Big)\nonumber\\
=&\ \sum_{k=0}^{\lambda-1}\Big(\sup_{s\ne t\in [0, 1] }\dfrac{|f_k(t)-f_k(s)|}{|t-s|}|t-s|^{1-N\al}+\sup_{ t\in [0, 1] }|f_k(t)|  \Big)\nonumber\\
\le&\ \sum_{k=0}^{\lambda-1}\Big(\sup_{s\ne t\in [0, 1] }\dfrac{|f_k(t)-f_k(s)|}{|t-s|}+\sup_{ t\in [0, 1] }|f_k(t)|  \Big) \hspace{2cm} (\text{by $1-N\al\geq 0$})\nonumber \\
=&\ \sum_{k=0}^{\lambda-1}\Big(\sup_{u\in [0, 1] }|f'_k(u)|+\sup_{ t\in [0, 1] }|f_k(t)|  \Big) \hspace{2cm} (\text{by $f_k\in \mathcal{C}^1_1(\RR,\RR)$  }  )\nonumber\\
\le&\ \sum_{k=0}^{\lambda-1}\Big(\sup_{u\in [0, 1] }(|f'_k(u)|+|f_k(u)|)+\sup_{ t\in [0, 1] }(|f'_k(t)|+|f_k(t)|)  \Big)\nonumber\\
=&\ 2\sum_{k=0}^{\lambda-1}\sup_{ t\in [0, 1] }(|f'_k(t)|+|f_k(t)|)\nonumber\\
\le&\ 2\|f\|_{1, 1}<\infty \hspace{2cm} (\text{by (\ref{eq:scal0})}).\mlabel{eq:sca6}
\end{align}
Combining the above bounds, we conclude
\begin{align*}
&\ \Big|\int_0^1 f(\lambda t)Z(t)\,dX(t)\Big|\\
=&\ \Big|\sum_{k=0}^{\lambda-1}\int_0^1f(t+k)Z_{\lambda, k}(t)\,dX_{\lambda, k}(t)\Big|\hspace{2cm} (\text{by (\ref{eq:scal2})})\\
\le&\ \sum_{k=0}^{\lambda-1} \Big|\int_0^1f(t+k)Z_{\lambda, k}(t)\,dX_{\lambda, k}(t)\Big| \\
=&\ \sum_{k=0}^{\lambda-1} {\Big|\int_0^1f(t+k)Z_{\lambda, k}(t)\,dX_{\lambda, k}(t) - \int_0^0 f(t+k)Z_{\lambda, k}(t)\,dX_{\lambda, k}(t) \Big|\over 1^{\al}} \\
\le&\ \sum_{k=0}^{\lambda-1}C_{\alpha, T}\|f_k\Z_{\lambda, k}\|_{\X_{\lambda, k},\al}\|\X_{\lambda, k}\|_{\al}
\hspace{2cm} (\text{by Proposition~\ref{thm:inte2}})\\
\le&\ \sum_{k=0}^{\lambda-1}C_{\alpha, T}\|f_k\|_{\mathcal{C}^{N\al}}\|\Z_{\lambda, k}\|_{\X_{\lambda, k},\al}\|\X_{\lambda, k}\|_{\al} \hspace{2cm} (\text{by (\ref{eq:regu1})})\\
\le&\ \sum_{k=0}^{\lambda-1}C_{\alpha, T}\|f_k\|_{\mathcal{C}^{N\al}}\|\Z\|_{\X,\al}(\lambda^{-\al}\|X^1\|_{\al}+\dots+\lambda^{-N\al}\|X^N\|_{N\al})\hspace{1cm}(\text{by (\ref{eq:sca3}) and (\ref{eq:sca4})})\\
=&\ C_{\alpha, T}\lambda^{-\al}\Big(\sum_{k=0}^{\lambda-1}\|f_k\|_{\mathcal{C}^{N\al}}\Big)\|\Z\|_{\X,\al}(\|X^1\|_{\al}+\dots+\lambda^{-(N-1)\al}\|X^N\|_{N\al})\\
\le&\ 2C_{\alpha, T}\lambda^{-\al}\|f\|_{1, 1}\|\Z\|_{\X,\al}(\|X^1\|_{\al}+\dots+\lambda^{-(N-1)\al}\|X^N\|_{N\al})\hspace{1cm}(\text{by (\ref{eq:sca6})})\\
\le&\ C_{\alpha, T}\lambda^{-\al}\|f\|_{1, 1}\|\Z\|_{\X,\al}\|\X\|_{\al}
 \hspace{0.5cm}(\text{by $\lambda^{-\al},\dots, \lambda^{-(N-1)\al}\le 1 $ from $\lambda \in \ZZ_{\geq 1}$}).
\end{align*}
This completes the proof.
\end{proof}

\section{Well-posedness}\label{ss:sec4}
In this section, we skillfully reformulate the linear stochastic heat equation~(\ref{eq:heat}) as
\begin{equation}
dh = (\partial_x^2 - 1)h dt + \eta dW(t),\quad \forall t\in [0,1], \, x\in [0, 2\pi]. \mlabel{eq:spde1}
\end{equation}
Here $(W_t)_{t\in [0,1]}$ denotes a standard cylindrical Wiener process on $L^2([0,2\pi], \mathbb{R}^d)$,
which serves as a model for space-time white noise.
Based on this formulation, we first lift the stationary solution of~(\ref{eq:spde1}) to a rough path and then provide a rigorous interpretation of~(\ref{eq:bspde}) within the rough path framework. All results in this section are understood in the pathwise sense.

\subsection{Gaussian rough paths and definition of solutions}
Let $\al \in (0, \frac{1}{2})$ and fix $t\in [0,1]$.
According to~\cite[Lemma 3.1]{Hair11}, the stochastic process $h_t:[0, 2\pi] \to \RR^d$, defined by~(\ref{eq:spde1}), is a centered Gaussian process whose covariance function has finite 1-variation.
Moreover, by~\cite[Proposition 2.5]{GOS20}, the path $H_t:=h_t$ admits a canonical lift to an
$\al$-H\"{o}lder weakly geometric rough path $$\Ha_t:=\bh\in \DA.$$
Here canonical means that the iterated integrals $H_t^i$, $1\leq i\leq N$, are defined by
\[H_t^i(x, y):=\lim_{\varepsilon\to 0}\sum_{1\le j_1, \ldots, j_n\le d}\left(\int_{x<x_{j_1}<\cdots<x_{j_n}<y}dH_t^{\varepsilon}(x_{j_1})\cdots dH_t^{\varepsilon}(x_{j_n})\right) e_{j_1}\otimes \cdots \otimes e_{j_n},\]
where $H_t^{\varepsilon}$ is any suitable sequence of smooth approximations to $H_t$ and $\{e_1,\ldots, e_d\}$ denotes the canonical basis of $\RR^d$.
For consistency with the notation in Section~\mref{ss:sec2}, we use $H_t$ to denote $h_t$ in this context.
The following is the concept of weak solution to (\ref{eq:bspde}).

\begin{definition}~\cite[Definition 3.2]{Hair11}
A continuous stochastic process $u: [0,1] \times [0,2\pi] \to \mathbb{R}^d $ is called a {\bf weak solution} to~(\ref{eq:bspde}) if the following conditions are satisfied:
\begin{enumerate}
\item
The process $v := u - h $ belongs to
\[
\mathcal{C}\Big([0,1], \mathcal{C}([0,2\pi], \mathbb{R}^d)\Big) \cap L^1\Big([0,1], \mathcal{C}^1([0,2\pi], \mathbb{R}^d)\Big).
\]

\item
For every smooth periodic test function $\varphi: [0,2\pi] \to \mathbb{R}$, the following identity holds almost surely:
\begin{align}
\langle v_t, \varphi \rangle &= \langle u_0 - h_0, \varphi \rangle + \int_0^t \langle (\partial_x^2 - 1)\varphi, v_s \rangle \, ds + \int_0^t \langle \varphi, g(u_s)\partial_x v_s \rangle \, ds \nonumber \\
&\quad + \int_0^t \int_0^{2\pi} \varphi(x) g\big(v_s(x) + H_s(x)\big) \, dH_s(x) \, ds + \int_0^t \langle \varphi, \hat{f}(u_s) \rangle \, ds,
\mlabel{eq:solu1}
\end{align}
where $\hat{f}(u_s) := f(u_s) + u_s $.
\end{enumerate}
\mlabel{def:solu1}
\end{definition}

\begin{remark}
For the double integral
$$ \int_0^t\int_0^{2\pi} \varphi(x)g(v_s(x)+H_s(x))d\,H_s(x)ds,$$
the inner integral is well-defined as a rough integral.
Moreover, the outer integral is also well-defined because the map
$$ s\mapsto \int_0^{2\pi} \varphi(x)g\Big(v_s(x)+H_s(x) \Big)d\,H_s(x)$$
is continuous with respect to $s$.
\mlabel{rk:roughint}
\end{remark}

The weak solution introduced in Definition~\ref{def:solu1} admits an equivalent formulation in the mild sense~\cite[Proposition 3.5]{Hair11}. Let $(S_t)_{t\in [0,1]}$
denote the heat semigroup on $[0,2\pi]$ subject to periodic boundary conditions. The associated heat kernel $p_t:[0, 2\pi] \to \RR$
is defined as the unique $2\pi$-periodic function such that, for every continuous function $u_0:[0, 2\pi]\to \RR^d$,
the semigroup action is given by
\begin{equation*}
(S_t u_0)(x)=\int_0^{2\pi}p_t(x-y)u_0(y)dy.
\end{equation*}
We now proceed to formulate the corresponding notion of a mild solution to~(\ref{eq:bspde}).

\begin{definition}\cite{Hair11}
The process
$$v:=u-h \in \mathcal{C}\Big([0, 1], \mathcal{C}([0,2\pi], \RR^d)\Big) \cap L^1\Big([0, 1],\mathcal{C}^1([0,2\pi], \RR^d) \Big)$$
is called {\bf a mild solution} to~(\ref{eq:bspde}) if $v$ satisfies the same conditions as in Definition~\ref{def:solu1}, but with~(\ref{eq:solu1}) replaced by the identity
\begin{align}
v_t(x)=&\ \big(S_t(u_0-h_0)\big)(x)+\int_0^t \big(S_{t-s}(g(u_s)\partial_xv_s+\hat{f}(u_s))\big)(x)ds \nonumber\\
&\ +\int_0^t\int_0^{2\pi} p_{t-s}(x-y)g(u(s, y))d\,H_s(y)ds.
\mlabel{eq:solu2}
\end{align}
\end{definition}

\subsection{Existence and uniqueness}
The following result provides an upper bound, which is a straightforward modification of~\cite[Lemma 3.8]{Hair11}.

\begin{lemma}
For any $s\in [0, 1]$, let $\al \in (0, \frac{1}{2})$, $\Ha_s \in \DA$ above $H_s$ and
$$\Z = (Z^0, \ldots, Z^{N-1})\in \mathcal{C}_{\Ha_s}^{\alpha}([0, 2\pi], \mathcal{L}(\RR^d, \RR^m))$$
above $Z$. Then for each $s,t\in [0,1]$,
\begin{equation*}
\Big|\int_0^{2\pi} \partial_xp_t(x-y)Z_y\,d\,H_s(y)\Big|\le C_{\alpha}t^{{\al \over 2}-1}\|\Z\|_{\Ha_s,\al}\|\Ha_s\|_{\al}
\end{equation*}
uniformly for $x\in [0, 2\pi]$ , where $C_{\alpha}\in \RR$ is independent of $H$ and $Z$.
\mlabel{lem:ex}
\end{lemma}

\begin{proof}
Notice that
\begin{align*}
\partial_xp_t(x)=&\ -\sum_{n\in \mathbb{Z}} {x\over \sqrt{2\pi}t^{3/2}}\exp\Big(-{(x-2\pi n)^2\over 2t}  \Big)   \\
=&\ -\sum_{n\in \mathbb{Z}}{1\over \sqrt{2\pi}t}{x\over \sqrt{t}}\exp\Big\{-{1\over 2}\Big({x\over \sqrt{t}}-{2\pi n\over \sqrt{t}}\Big)^2\Big\}.
\end{align*}
Let
\begin{equation*}
f_t:\RR\to \RR, \quad y \mapsto -\sum_{n\in \mathbb{Z}}{1\over \sqrt{2\pi}}y\cdot \exp\Big(-{1\over 2}(y-{2\pi n\over \sqrt{t}})^2\Big).
\end{equation*}
Then
\begin{equation}
\partial_xp_t(x)={1\over t}f_t\Big({x\over \sqrt{t}}\Big).
\mlabel{eq:ex1}
\end{equation}
It follows from~\cite[Lemma~3.8]{Hair11} that $\sup_{t\in (0, 1]}\|f_t\|_{1, 1}< \infty$. Hence
\begin{align*}
&\ \Big|\int_0^{2\pi} \partial_xp_t(x-y)Z_y\,d\,H_s(y)\Big| \\
=&\ \Big|\int_0^{2\pi} {1\over t}f_t\Big({x-y\over \sqrt{t}}\Big)Z_y\,d\,H_s(y)\Big|\hspace{2cm} (\text{by (\ref{eq:ex1})})\\
%
%
\le&\ {1\over t}C_{\al}\Big({1\over\sqrt{t}}\Big)^{-\al}\|f_t\|_{1, 1}\|\Z\|_{\Ha_s,\al}\|\Ha_s\|_{\al}\hspace{1.5cm} (\text{by (\ref{eq:sca1})})\\
\le&\ C_{\al}\Big(\sup_{t\in (0, 1]}\|f_t\|_{1, 1}\Big)t^{{\al \over 2}-1}\|\Z\|_{\Ha_s,\al}\|\Ha_s\|_{\al}\\
=:&\ C_{\al}t^{{\al \over 2}-1}\|\Z\|_{\Ha_s,\al}\|\Ha_s\|_{\al}.
\end{align*}
This completes the proof.
\end{proof}

We are now in a position to present one of the central results of this section, which asserts the existence and uniqueness of local solutions in the pathwise sense.

\begin{theorem}
Let $\beta \in (0, \frac{1}{2}) $ and $u_0\in \mathcal{C}^{\beta}([0,2\pi],\RR^d)$.
Then, for almost every realization of the driving process $H$ and a time $T\in (0,1]$ to be small enough,
the equation~(\ref{eq:bspde})
admits a unique mild solution $u$ in $\mathcal{C}\Big([0, T],\mathcal{C}^{\beta}([0,2\pi],\RR^d)\Big).$
\mlabel{thm:eu1}
\end{theorem}

\begin{proof}
Fix
\begin{equation}
N:= \lfloor \frac{1}{\beta}\rfloor, \quad \al \in (\frac{1}{N+1},\beta), \quad \Ha_t:=\bh\in \mathcal{D}^\alpha([0,2\pi], \RR^d)\, \text{ above }\, H_t,\quad \forall t\in [0, T].
\mlabel{eq:c11t}
\end{equation}
Note that
\begin{equation}
2\al-\beta\overset{(\ref{eq:c11t})}{>}\frac{2}{N+1}-\frac{1}{N}=\frac{N-1}{N(N+1)}\ge 0.
\mlabel{eq:bealpha}
\end{equation}
We prove this theorem by using the Banach fixed point method. First, we define a Banach space $(\mathcal{C}^1_T, \|\cdot\|_{1, T})$ by setting
\begin{equation}
\mathcal{C}^1_T:= \{v:[0, T]\to  \mathcal{C}^1([0, 2\pi], \RR^d) \mid \|v\|_{1, T}<\infty  \}, \quad \|v\|_{1, T}:=\sup_{0\leq t\le T}\|v_t\|_{\mathcal{C}^1}.
\mlabel{eq:c1t}
\end{equation}
Denote
\begin{equation}
U_t :=S_t(u_0-h_0), \quad w_t:=u_t-H_t-U_t,\quad v_t:=u_t-H_t = w_t+U_t, \quad \forall t\in [0, T].
\mlabel{eq:uwv}
\end{equation}
Then (\ref{eq:solu2}) is transformed into
\begin{align*}
w_t(x)=&\ \int_0^t \Big(S_{t-s}(g(u_s)(\partial_xw_s+\partial_xU_s)+\hat{f}(u_s))\Big)(x)ds \nonumber\\
&\ +\int_0^t\int_0^{2\pi} p_{t-s}(x-y)g\Big(w(s, y)+H_s(y)+U_s(y)\Big)d\,H_s(y)ds.
\end{align*}
This motivates us to define a map
\begin{equation*}
\mathcal{M}_{T, H}:\mathcal{C}^1_T \to \mathcal{C}^1_T, \quad w\mapsto \mathcal{M}_{T, H}(w)
\end{equation*}
given by
\begin{align}
(\mathcal{M}_{T, H}w)(t, x):=&\ \int_0^t \bigg(S_{t-s}\Big(g(u_s)(\partial_xw_s+\partial_xU_s)+\hat{f}(u_s)\Big)\bigg)(x)ds \nonumber\\
&\ +\int_0^t\int_0^{2\pi} p_{t-s}(x-y)g\Big(w(s, y)+H_s(y)+U_s(y)\Big)d\,H_s(y)ds.     \mlabel{eq:solu3}
\end{align}
To handle the two integrals above, we define the following
two mappings
$$\mathcal{M}_{T, H}^{(1)}, \, \mathcal{M}_{T, H}^{(2)}:\mathcal{C}^1_T \to \mathcal{C}^1_T$$
given by
\begin{align}
(\mathcal{M}_{T, H}^{(1)}w)(t, x):=&\ \int_0^t \bigg(S_{t-s}\Big(g(u_s)(\partial_xw_s+\partial_xU_s)+\hat{f}(u_s)\Big)\bigg)(x)ds,\mlabel{eq:solu18} \\
(\mathcal{M}_{T, H}^{(2)}w)(t, x):=&\ \int_0^t\int_0^{2\pi} p_{t-s}(x-y)g\Big(w(s, y)+H_s(y)+U_s(y)\Big)d\,H_s(y)ds.\mlabel{eq:solu19}
\end{align}
The remainder of the proof is divided into the following two steps to establish the contractivity of $\mathcal{M}_{T, H}^{(1)}$ and $\mathcal{M}_{T, H}^{(2)}$.

\noindent {\bf Step 1. Contractivity of $\mathcal{M}_{T, H}^{(1)}$.}
Notice that
\begin{equation}
S_t:L^{\infty}([0,2\pi], \RR^d) \rightarrow \mathcal{C}^1([0,2\pi], \RR^d),\quad \forall t\in [0,T] \mlabel{eq:semigp}
 \end{equation}
is a linear operator with an upper bound
\begin{equation}
\|S_t \| \leq Ct^{-1/2}. \mlabel{eq:stbo}
\end{equation}

Let $w, \lbar{w} \in\mathcal{C}^1_T$ and choose $K>1$ such that
\begin{equation}
\|w\|_{1, T}\le K,\quad \|\lbar{w}\|_{1, T}\le K,\quad \|U\|_{1, T}\le K.
\mlabel{eq:solu11}
\end{equation}
Denote
\begin{equation}
\Hh :=\sup_{0\leq t\le 1}(\|H^1_t\|_{\mathcal{C}^{\al}}+\|H^2_t\|_{2\al}+\dots+\|H^N_t\|_{N\al})\overset{(\ref{eq:note1}), (\ref{eq:note4})}{=}\sup_{0\leq t\le 1}(\|H_t\|_{\infty}+\|\Ha_t\|_{\al}).
\mlabel{eq:Hh}
\end{equation}
Then
\begin{align}
&\ \|\mathcal{M}_{T, H}^{(1)}w\|_{1, T}\nonumber\\
=&\ \sup_{0\leq t\le T}\|\mathcal{M}_{T, H}^{(1)}w_t\|_{\mathcal{C}^1}\nonumber\\
=&\ \sup_{0\leq t\le T}\Big\|\int_0^t \Big(S_{t-s}(g(u_s)(\partial_xw_s+\partial_xU_s)+\hat{f}(u_s))\Big)ds \Big\|_{\mathcal{C}^1}\nonumber\\
\le&\ \sup_{0\leq t\le T}\int_0^t\Big\|S_{t-s}\Big(g(u_s)(\partial_xw_s+\partial_xU_s)+\hat{f}(u_s)\Big) \Big\|_{\mathcal{C}^1}ds\nonumber\\
\le&\ \sup_{0\leq t\le T}\int_0^t\|S_{t-s}\|\,\|g(u_s)(\partial_xw_s+\partial_xU_s)+\hat{f}(u_s)\|_{\infty}ds\nonumber\\
=&\ \sup_{0\leq t\le T}\int_0^t\|S_{t-s}\|\,\|g(u_s)(\partial_xw_s+\partial_xU_s)+f(u_s)+u_s\|_{\infty}ds\hspace{1cm} (\text{by $\hat{f}(u)=f(u)+u$})\nonumber\\
=&\ \sup_{0\leq t\le T}\int_0^t\|S_{t-s}\|\,\|g(u_s)(\partial_xw_s+\partial_xU_s)+f(u_s)+w_s+H_s+U_s\|_{\infty}ds\hspace{0.4cm} (\text{by $u=w+H+U$})\nonumber\\
\le&\ \sup_{0\leq t\le T}\int_0^t\|S_{t-s}\|\Big(\|g(u_s)\|_{\infty}(\|\partial_xw_s\|_{\infty}+
\|\partial_xU_s\|_{\infty})+\|f(u_s)\|_{\infty}+\|w_s\|_{\infty}+\|H_s\|_{\infty}+\|U_s\|_{\infty}\Big)ds\nonumber\\
\le&\ \sup_{0\leq t\le T}\int_0^t\|S_{t-s}\|\Big(\|g\|_{\infty}(\|w\|_{1, T}+K)+\|f\|_{\infty}+\|w\|_{_{1, T}}+\Hh+K\Big)ds\nonumber\\
\le&\ \sup_{0\leq t\le T}\int_0^tC(t-s)^{-1/2}\Big(\|g\|_{\infty}(K+K)+\|f\|_{\infty}+K+\Hh+K \Big)ds \hspace{1cm} (\text{by (\ref{eq:stbo}) and~(\ref{eq:solu11})})\nonumber\\
\le&\  \sup_{0\leq t\le T}\int_0^tC(t-s)^{-1/2}2(1+\|g\|_{\infty}+\|f\|_{\infty})(1+K)(1+\Hh)ds\nonumber\\
=&\ \sup_{0\leq t\le T}4Ct^{1/2}(1+\|g\|_{\infty}+\|f\|_{\infty})(1+K)(1+\Hh)\nonumber\\
\le&\ \Big(4C(1+\|g\|_{\infty}+\|f\|_{\infty})(1+K)\Big)(1+\Hh)T^{1/2}\nonumber\\
=&\ \Big(4C(1+\|g\|_{\infty}+\|f\|_{\infty})(1+K)\Big)(1+\Hh)T^{(1-\beta)/2}T^{\beta/2}\nonumber\\
\le&\ \Big(4C(1+\|g\|_{\infty}+\|f\|_{\infty})(1+K)\Big)(1+\Hh)T^{\beta/2}\hspace{1cm} (\text{by $T\le 1$ and $\beta\le 1 $})\nonumber\\
=:&\ C(1+\Hh)T^{\beta/2}. \mlabel{eq:solu4}
\end{align}
From the above calculation, if $f$ and $g$ are bounded, then then the constant
$C$ can be chosen proportional to $K$.
To establish the contractivity of the map
$\mathcal{M}_{T, H}^{(1)} $, we observe the following
\begin{align}
(\mathcal{M}_{T, H}^{(1)}w-\mathcal{M}_{T, H}^{(1)}\lbar{w})_t=&\ \int_0^t \bigg(S_{t-s}\Big(g(u_s)(\partial_xw_s+\partial_xU_s)+\hat{f}(u_s)\Big)\bigg)ds\nonumber\\
&\ -  \int_0^t \bigg(S_{t-s}\Big(g(\lbar{u}_s)(\partial_x\lbar{w}_s+\partial_xU_s)+\hat{f}(\lbar{u}_s)\Big)\bigg)ds \hspace{2cm} (\text{by (\ref{eq:solu18})}) \nonumber \\
=&\ \int_0^t \bigg(S_{t-s}\Big(g(u_s)(\partial_xw_s-\partial_x\lbar{w}_s+\partial_x\lbar{w}_s+\partial_xU_s)+\hat{f}(u_s)\Big)\bigg)ds\nonumber\\
&\ -  \int_0^t \bigg(S_{t-s}\Big(g(\lbar{u}_s)(\partial_x\lbar{w}_s+\partial_xU_s)+\hat{f}(\lbar{u}_s) \Big)\bigg)ds  \nonumber \\
=&\ \int_0^t S_{t-s}\Big(g(u_s) (\partial_xw_s-\partial_x\lbar{w}_s)+\hat{f}(u_s)-\hat{f}(\lbar{u}_s)\Big)ds\nonumber\\
&\ -  \int_0^t S_{t-s}\Big((g(u_s)-g(\lbar{u}_s))(\partial_x\lbar{w}_s+\partial_xU_s)\Big)ds .\mlabel{eq:fina3}
\end{align}
Further,
\begin{align}
\|g(u_s)-g(\lbar{u}_s)\|_{\infty}\le&\ \|g'\|_{\infty}\|u_s-\lbar{u}_s\|_{\infty} \hspace{2cm}(\text{by differential mean value theorem}) \nonumber\\
=&\ \|g'\|_{\infty}\|(w_s+H_s+U_s)-(\lbar{w}_s+H_s+U_s )\|_{\infty} \hspace{2cm}(\text{by (\ref{eq:uwv})}) \nonumber\\
=&\ \|g'\|_{\infty}\|w_s-\lbar{w}_s\|_{\infty}\nonumber\\
\le&\ \|g'\|_{\infty}\|w_s-\lbar{w}_s\|_{\mathcal{C}^1}\nonumber \\
\le&\ C \|w_s-\lbar{w}_s\|_{\mathcal{C}^1} \hspace{2cm}(\text{by taking a constant $C$}),\mlabel{eq:fina4}
\end{align}
and similarly,
\begin{equation}
\|\hat{f}(u_s)-\hat{f}(\lbar{u}_s)\|_{\infty}\le C \|w_s-\lbar{w}_s\|_{\mathcal{C}^1}.\mlabel{eq:fina5}
\end{equation}
The contractivity of the map $\mathcal{M}_{T, H}^{(1)} $ is established through the following argument:
\begin{align}
&\ \|\mathcal{M}_{T, H}^{(1)}w-\mathcal{M}_{T, H}^{(1)}\lbar{w}\|_{1, T}\nonumber\\
=&\ \sup_{0\leq t\le T}\|\mathcal{M}_{T, H}^{(1)}w_t-\mathcal{M}_{T, H}^{(1)}\lbar{w}_t\|_{\mathcal{C}^1}\nonumber\\
=&\ \sup_{0\leq t\le T}\Big\|\int_0^t S_{t-s}\Big(g(u_s) (\partial_xw_s-\partial_x\lbar{w}_s)+\hat{f}(u_s)-\hat{f}(\lbar{u}_s)\Big)ds\nonumber\\
&\ -  \int_0^t S_{t-s}\Big((g(u_s)-g(\lbar{u}_s))(\partial_x\lbar{w}_s+\partial_xU_s)\Big)ds\Big\|_{\mathcal{C}^1}\hspace{2cm} (\text{by (\ref{eq:fina3})})\nonumber\\
\le&\ \sup_{0\leq t\le T}\Big\|\int_0^t S_{t-s}\Big(g(u_s) (\partial_xw_s-\partial_x\lbar{w}_s)+\hat{f}(u_s)-\hat{f}(\lbar{u}_s)\Big)ds \Big\|_{\mathcal{C}^1}\nonumber\\
&\ +\sup_{0\leq t\le T}\Big\|\int_0^t S_{t-s}\Big((g(u_s)-g(\lbar{u}_s))(\partial_x\lbar{w}_s+\partial_xU_s)\Big)ds \Big\|_{\mathcal{C}^1} \nonumber\\
\le&\ \sup_{0\leq t\le T}\int_0^t \|S_{t-s}\|\,\Big(\|g(u_s)\|_{\infty} \|\partial_xw_s-\partial_x\lbar{w}_s\|_{\infty}+\|\hat{f}(u_s)-\hat{f}(\lbar{u}_s)\|_{\infty}\Big)ds \hspace{0.5cm} (\text{by (\ref{eq:semigp})}) \nonumber\\
&\ +\sup_{0\leq t\le T}\int_0^t \|S_{t-s}\|\,\Big(\|g(u_s)-g(\lbar{u}_s)\|_{\infty}\|\partial_x\lbar{w}_s+\partial_xU_s\|_{\infty}\Big)ds  \nonumber\\
\le&\ \sup_{0\leq t\le T}\int_0^t \|S_{t-s}\|\,\Big(\|g(u_s)\|_{\infty} \|w_s-\lbar{w}_s\|_{\mathcal{C}^1}+C\|w_s-\lbar{w}_s\|_{\mathcal{C}^1}\Big)ds \nonumber\\
&\ +\sup_{0\leq t\le T}\int_0^t \|S_{t-s}\|\,\Big(C\|w_s-\lbar{w}_s\|_{\mathcal{C}^1}\|\partial_x\lbar{w}_s+\partial_xU_s\|_{\infty}\Big)ds \hspace{0.5cm} (\text{by (\ref{eq:fina4}) and (\ref{eq:fina5})}) \nonumber\\
=&\ \bigg(\sup_{0\leq t\le T}\int_0^t \|S_{t-s}\|\,\Big(\|g(u_s)\|_{\infty}+C\Big)ds
+\sup_{0\leq t\le T}\int_0^t \|S_{t-s}\|\,\Big(C\|\partial_x\lbar{w}_s+\partial_xU_s\|_{\infty}\Big)ds\bigg)\|w_s-\lbar{w}_s\|_{\mathcal{C}^1}\nonumber\\
\le&\ C(1+\Hh)T^{\beta/2}\|w-\lbar{w}\|_{1, T}. \mlabel{eq:solu5}
\end{align}
Here, the final inequality follows from (\ref{eq:c1t}) and an argument analogous to that of (\ref{eq:solu4}), by replacing
\[
1 \,\text{ with }\, \partial_xv_s+\partial_xU_s, \quad C\,\text{ with }\, \hat{f}(u_s),\quad C(\partial_x\lbar{v}_s+\partial_xU_s)\,\text{ with }\, g(u_s)(\partial_xv_s+\partial_xU_s)+\hat{f}(u_s).
\]
Based on the estimates in (\ref{eq:fina4}), (\ref{eq:fina5}) and (\ref{eq:solu5}), the constant
$C$ can be taken proportional to $K$, assuming that $g$, $Dg$, $f$ and $Df$ are bounded. \vskip 0.1in

\noindent {\bf Step 2. Contractivity of $\mathcal{M}_{T, H}^{(2)}$.}
As noted in Remark~\mref{rk:roughint}, the inner integral on the right-hand side of~(\ref{eq:solu19}) can be interpreted as a rough integral.
Since $$(H_s, \id, \underbrace{0, \ldots, 0}_{N-2})$$ is an ${\bf H}_s$-controlled rough path, by applying Definition~\mref{def:regu} with
$$\varphi(\cdot):= g(w(s, x)+\cdot+U(s, x)):\RR^d\to  \mathcal{L}(\RR^d, \RR^d),$$
we conclude that
$$\Z_s := (Z^0_s, \ldots, Z^{N-1}_s) \in  \mathcal{C}_{\Ha_s}^{\alpha}([0, 2\pi], \mathcal{L}(\RR^d, \RR^d))$$ is a controlled rough path,
where
\begin{equation}
Z^0_s(x):= g(w(s, x)+H(s, x)+U(s, x))
\mlabel{eq:16}
\end{equation}
and
\begin{equation}
Z^r_s(x):=\sum_{j=1}^{N-1}{1\over j!}D^jg(w(s, x)+H(s, x)+U(s, x))\Big(\sum_{i_1+\dots+i_j=r}\big(H_s^{i_1}(x)\otimes\dots \otimes H_s^{i_j}(x)\big)\circ \delta_j \mid _{\RR^{\otimes r}} \Big)
\mlabel{eq:solu6}
\end{equation}
for $r =1, \dots, N-1$. By Proposition~\ref{pro:regu}, there exists a constant $C\in \RR$ such that
\begin{align}
\|\Z_s\|_{\Ha_s,\al}\le&\ C\Big(\sum_{i=0}^N\|D^i\f\|_{\infty}\Big)\Big(\sum_{i=1}^{N-1}\|H^i\|_{\C}^l\Big)(1+\Hh)^k\nonumber\\
\le&\ C\Big(\sum_{i=0}^N\|D^ig\|_{\infty}\Big)\Big(\sum_{i=1}^{N-1}\|H^i\|_{\C}^l\Big)(1+\Hh)^k\nonumber\\
\le&\ C \Big(\sum_{i=0}^N\|D^ig\|_{\infty}+\|w_s\|_{2\al}+\|U_s\|_{2\al}\Big)\Big(\sum_{i=1}^{N-1}\|H^i\|_{\C}^l\Big)(1+\Hh)^k \nonumber\\
\le&\ C \Big({\sum_{i=0}^N\|D^ig\|_{\infty}+\| w_s\|_{2\al}\over \|w_s\|_{2\al}}  +1\Big)(\|w_s\|_{2\al}+\|U_s\|_{2\al})\Big(\sum_{i=1}^{N-1}\|H^i\|_{\C}^l\Big)(1+\Hh)^k \nonumber\\
=:&\ C (\|w_s\|_{2\al}+\|U_s\|_{2\al})\Big(\sum_{i=1}^{N-1}\|H^i\|_{\C}^l\Big)(1+\Hh)^k. \mlabel{eq:fina0}
\end{align}
By the property of the heat semigroup, we have~\cite{Hair11}
\begin{equation}
\|U_s\|_{2\al}\le Cs^{-(2\al-\beta)/2}
\mlabel{eq:fina1}
\end{equation}
for some constant $C$ and so
\begin{align}
\|\Z_s\|_{\Ha_s,\al}\le&\ C (\|w_s\|_{2\al}+Cs^{-(2\al-\beta)/2})\Big(\sum_{i=1}^{N-1}\|H^i\|_{\C}^l\Big)(1+\Hh)^k \hspace{2cm} (\text{by (\ref{eq:fina0}) and~(\mref{eq:fina1})})\nonumber\\
\le&\ C (\|w_s\|_{2\al}+C)(1+s^{-(2\al-\beta)/2})\Big(\sum_{i=1}^{N-1}\|H^i\|_{\C}^l\Big)(1+\Hh)^k\nonumber\\
=&\ C \Big(\sum_{i=1}^{N-1}\|H^i\|_{\C}^l\Big)(1+\Hh)^k(1+s^{-(2\al-\beta)/2}) \hspace{1cm} (\text{by setting $C:=C(\|w_s\|_{2\al}+C)$}).\mlabel{eq:solu7}
\end{align}
From the calculation in (\ref{eq:fina0}), the constant $C$ can be taken proportional to $K$, provided that
$D^ig$, for $i=0,\ldots, N$, are all bounded.

To estimate $\|\mathcal{M}_{T, H}^{(2)}v\|_{1, T}$, we observe that
\begin{align}
&\ \Big|\partial_x\int_0^{2\pi} p_{t-s}(x-y)Z_s(y)\,d\,H_s  \Big|\\
\le&\ C(t-s)^{{\al \over 2}-1}\|\Z_s\|_{\Ha,\al}\|\Ha_s\|_{\al} \hspace{2.5cm} \text{(by Lemma~\ref{lem:ex})}\nonumber\\
\leq &\ C (t-s)^{{\al \over 2}-1} \Big(\sum_{i=1}^{N-1}\|H^i\|_{\C}^l\Big)(1+\Hh)^k(1+s^{-{2\al-\beta\over 2}})(1+\Hh) \hspace{1.5cm} (\text{by~(\ref{eq:Hh}) and~(\ref{eq:solu7})})\nonumber\\
=&\ C\Big(\sum_{i=1}^{N-1}\|H^i\|_{\C}^l\Big)(1+\Hh)^{k+1}(1+s^{-{2\al-\beta\over 2}})(t-s)^{{\al \over 2}-1}. \mlabel{eq:solu12}
\end{align}
From this, we obtain the desired bound
\begin{align}
&\ \|\mathcal{M}_{T, H}^{(2)}v\|_{1, T}\nonumber\\
=&\ \sup_{0\leq t\le T}\|\mathcal{M}_{T, H}^{(2)}v_t\|_{\mathcal{C}^1}\nonumber\\
=&\ \sup_{0\leq t\le T}\Big\|\int_0^t\int_0^{2\pi} p_{t-s}(\cdot-y)g(v(y, s)+H_s(y)+U_s(y))d\,H_s(y)ds\Big\|_{\mathcal{C}^1}\nonumber\\
=&\ \sup_{0\leq t\le T}\Big\|\int_0^t\int_0^{2\pi} p_{t-s}(\cdot-y)Z_s(y) d\,H_s(y)ds\Big\|_{\mathcal{C}^1}
\hspace{2cm} (\text{by (\ref{eq:solu6})})\nonumber \\
\le&\ \sup_{0\leq t\le T}\int_0^t\Big\|\int_0^{2\pi} p_{t-s}(\cdot-y)Z_s(y) d\,H_s(y) \Big\|_{\mathcal{C}^1}ds\nonumber\\
=&\ \sup_{0\leq t\le T}\int_0^t\bigg(\Big\|\int_0^{2\pi} p_{t-s}(\cdot-y)Z_s(y) d\,H_s(y)\Big\|_{\infty}+\Big\|\int_0^{2\pi} p_{t-s}(\cdot-y)Z_s(y)d\,H_s(y)\Big\|_{1}\bigg)ds\hspace{1cm} (\text{by (\ref{eq:note2})})\nonumber\\
\le&\ \sup_{0\leq t\le T}C\int_0^t\Big\|\partial_x\int_0^{2\pi} p_{t-s}(\cdot-y)Z_s(y)d\,H_s(y) \Big\|_{\infty}ds \nonumber \\
&\ \hspace{1.5cm}(\text{Both of the above norms can be controlled by the infinite norm of the derivative})\nonumber\\
\le&\ \sup_{0\leq t\le T}C\int_0^t\Big(\sum_{i=1}^{N-1}\|H^i\|_{\C}^l\Big)(1+\Hh)^{k+1}(1+s^{-{2\al-\beta\over 2}})(t-s)^{{\al \over 2}-1}ds\hspace{2cm} (\text{by (\ref{eq:solu12})})\nonumber\\
\le&\ \sup_{0\leq t\le T}C\Big(\sum_{i=1}^{N-1}\|H^i\|_{\C}^l\Big)(1+\Hh)^{k+1}  (t^{\al \over 2}+t^{{-2\al+\beta+2\over 2}}) \hspace{2cm} (\text{by calculating the integral})\nonumber\\
=&\ \sup_{0\leq t\le T}C\Big(\sum_{i=1}^{N-1}\|H^i\|_{\C}^l\Big)(1+\Hh)^{k+1}(t^{2\al-\beta \over 2}+t^{{-\al+2\over 2}})t^{\beta-\al \over 2}\nonumber\\
\le&\ C\Big(\sum_{i=1}^{N-1}\|H^i\|_{\C}^l\Big)(1+\Hh)^{k+1}(T^{2\al-\beta \over 2}+T^{{-\al+2\over 2}})T^{\beta-\al \over 2}\hspace{1cm} (\text{by $2\al-\beta,\ -\al+2,\ \beta-\al \overset{(\ref{eq:c11t}),(\ref{eq:bealpha})}{>} 0$})\nonumber\\
\le&\ C\Big(\sum_{i=1}^{N-1}\|H^i\|_{\C}^l\Big)(1+\Hh)^{k+1}T^{\beta-\al \over 2}\hspace{2cm} (\text{by $T\le1$}). \mlabel{eq:solu8}
\end{align}

Let $\lbar{\Z}_s=(\lbar{Z}^0_s, \ldots, \lbar{Z}^{N-1}_s)$ be an $\Ha_s$-controlled rough path associated to $\lbar{w}_s$.
Then
\begin{align*}
\lbar{Z}^0_s-Z^0_s \overset{(\ref{eq:16})}{=}&\ g\Big(\lbar{w}(s, x)+H(s, x)+U(s, x)\Big)-g\Big(w(s, x)+H(s, x)+U(s, x)\Big)\\
=&\ \int_0^1 \frac{d}{d\lambda} g\Big(H_s(x)+U_s(x)+w_s(x)+\lambda(\lbar{w}_s(x)-w_s(x))\Big)(\lbar{w}_s(x)-w_s(x))d\lambda.
\end{align*}
Notice that $\lbar{\Z}_s-\Z_s$ is still an $\Ha_s$-controlled rough path.
Using a similar argument as in (\ref{eq:solu6}) and (\ref{eq:fina0}), and applying Definition~\mref{def:regu} with
$$\varphi(\cdot):= \int_0^1Dg\Big(\cdot +U_s(x)+w_s(x)+\lambda(\lbar{w}_s(x)-w_s(x))\Big)(\lbar{w}_s(x)-w_s(x))d\lambda:\RR^d\to  \mathcal{L}(\RR^d, \RR^d),$$
we obtain
\begin{align}
 \|\lbar{\Z}_s-\Z_s\|_{\Ha_s,\al}
\le&\ C  \Big(\sum_{i=0}^N\|D^i\f\|_{\infty}\Big)\Big(\sum_{i=1}^{N-1}\|H^i\|_{\C}^l\Big)(1+\Hh)^k \nonumber\\
\le&\ C\Big(\sum_{i=0}^N\int_0^1\|D^{i+1}g\|_{\infty}\|\lbar{w}-w\|_{\infty}d\lambda\Big)\Big(\sum_{i=1}^{N-1}\|H^i\|_{\C}^l\Big)(1+\Hh)^k\nonumber\\
=&\ C\Big(\sum_{i=0}^N\|D^{i+1}g\|_{\infty}\Big)\Big(\sum_{i=1}^{N-1}\|H^i\|_{\C}^l\Big)(1+\Hh)^k\|\lbar{w}-w\|_{\infty}  \nonumber\\
\le&\ C\Big(\sum_{i=0}^N\|D^{i+1}g\|_{\infty}\Big)\Big(\sum_{i=1}^{N-1}\|H^i\|_{\C}^l\Big)(1+\Hh)^k\|\lbar{w}-w\|_{1, T}\nonumber\\
\le&\ C\Big(\sum_{i=1}^{N-1}\|H^i\|_{\C}^l\Big)(1+\Hh)^k(1+s^{-{2\al-\beta\over 2}})\|\lbar{w}-w\|_{1, T}. \mlabel{eq:solu21}
\end{align}
Here, the last step follows from~(\ref{eq:solu7}) and the last three steps of~(\ref{eq:fina0}).
Further,
\begin{align}
&\ \Big|\partial_x\int_0^{2\pi} p_{t-s}(x-y)(\lbar{Z}_s-Z_s)\,d\,H_s  \Big| \nonumber\\
\le&\ C(t-s)^{{\al \over 2}-1}\|\lbar{\Z}_s-\Z_s\|_{\Ha,\al}\|\Ha_s\|_{\al} %
\hspace{1cm} \text{(by Lemma~\ref{lem:ex})}\nonumber\\
\le&\ C\Big(\sum_{i=1}^{N-1}\|H^i\|_{\C}^l\Big)(1+\Hh)^{k+1}(1+s^{-{2\al-\beta\over 2}})(t-s)^{{\al \over 2}-1}\|\lbar{w}-w\|_{1, T}
\hspace{1cm} (\text{by (\ref{eq:Hh}) and (\ref{eq:solu21})}),    \mlabel{eq:solu20}
\end{align}
and so
\begin{align}
&\ \|\mathcal{M}_{T, H}^{(2)}w-\mathcal{M}_{T, H}^{(2)}\lbar{w}\|_{1, T} \nonumber\\
\leq &\ \sup_{0\leq t\le T}C\int_0^t\Big\|\partial_x\int_0^{2\pi} p_{t-s}(\cdot-y)(\lbar{Z}_s-Z_s)(y)d\,H_s(y) \Big\|_{\infty}ds \hspace{1cm} (\text{by the first six steps of~(\ref{eq:solu8})})\nonumber \\
\le&\ \sup_{0\leq t\le T}C\int_0^t\Big(\sum_{i=1}^{N-1}\|H^i\|_{\C}^l\Big)(1+\Hh)^{k+1}(1+s^{-{2\al-\beta\over 2}})(t-s)^{{\al \over 2}-1}\|\lbar{w}-w\|_{1, T}ds\hspace{1cm} (\text{by (\ref{eq:solu20})})\nonumber \\
\le&\ C\Big(\sum_{i=1}^{N-1}\|H^i\|_{\C}^l\Big)(1+\Hh)^{k+1}T^{\beta-\al \over 2}\|\lbar{w}-w\|_{1, T} \hspace{1cm} (\text{by the last five steps of (\ref{eq:solu8})}). \mlabel{eq:solu9}
\end{align}
From the calculation in (\ref{eq:solu21}), it follows that $C$ is proportional to
$K$, provided that $Dg$, $D^2g$, $D^3g$ and $D^4g$ are all bounded.

In summary, combining~(\ref{eq:solu3}), (\ref{eq:solu5}) and~(\ref{eq:solu9}), we conclude
\begin{align*}
\|\mathcal{M}_{T, H}w-\mathcal{M}_{T, H}\lbar{w}\|_{1, T}
\le&\  \|\mathcal{M}_{T, H}^{(1)}w-\mathcal{M}_{T, H}^{(1)}\lbar{w}\|_{1, T}+\|\mathcal{M}_{T, H}^{(2)}w-\mathcal{M}_{T, H}^{(2)}\lbar{w}\|_{1, T}\\
\le& C(1+\Hh)T^{\beta/2}\|w-\lbar{w}\|_{1, T}+ C\Big(\sum_{i=1}^{N-1}\|H^i\|_{\C}^l\Big)(1+\Hh)^{k+1}T^{\beta-\al \over 2}\|w-\lbar{w}\|_{1, T}\\
\le&\ {1\over 2}\|w-\lbar{w}\|_{1, T}\hspace{1cm} (\text{by $T$ being small enough}).
\end{align*}
By applying the Banach fixed point theorem,
there is a unique $$u\in \mathcal{C}\Big([0, T],\mathcal{C}^{\beta}([0,2\pi],\RR^d)\Big)$$ that satisfies~(\ref{eq:bspde}). This completes the proof.
\end{proof}

The following result demonstrates that the unique solution to equation~(\ref{eq:bspde}) is global in the pathwise sense.

\begin{theorem}
Let $\beta \in (0, \frac{1}{2}) $ and $u_0\in \mathcal{C}^{\beta}([0,2\pi],\RR^d)$. Suppose that $f$ and $g$ are bounded, and all derivatives of $f$ and $g$ are bounded.
Then, for almost every realization of the driving process $H$, the equation~(\ref{eq:bspde}) has a unique mild solution $u$ in $\mathcal{C}\Big([0, 1],\mathcal{C}^{\beta}([0,2\pi],\RR^d)\Big)$.
\mlabel{thm:eu2}
\end{theorem}

\begin{proof}
{\bf (Existence).} Let $T$ be as in Theorem~\ref{thm:eu1}.
Notice that $T$ does not depend on the initial condition $u_0$.
By Theorem~\ref{thm:eu1}, we get a local solution on $[0, T]$.
Taking $u_T$ as a new initial condition, we obtain a solution to~(\ref{eq:bspde}) on $[T, 2T]$ by Theorem~\ref{thm:eu1} again.
Continuing this process, we conclude a solution $u$ to~(\ref{eq:bspde}) on $[0, 1]$ after finite steps.

{\bf (Uniqueness).} Let $\tilde u\in \mathcal{C}\Big([0, 1],\mathcal{C}^{\beta}([0,2\pi],\RR^d)\Big)$
be another required one. Define
$$\sigma:=\sup\Big\{t\ge 0 \mid u_t=\tilde u_t \text{ on }[0, 1] \Big\}.$$
Since $u$ and $\tilde u$ are continuous, we have $u_{\sigma}=\tilde u_{\sigma}$. Suppose for a contradiction that $\sigma <1$. Take $\varepsilon>0$ to be small enough such that $\varepsilon<T$ and $\sigma+\varepsilon<1$. Then
by the definition of $\sigma$,
\begin{equation}
u_{\sigma+{\varepsilon\over 2}}\ne \tilde u_{\sigma+{\varepsilon\over 2}}.
\mlabel{eq:solu13}
\end{equation}
By Theorem~\ref{thm:eu1}, there is a unique $u\in \mathcal{C}([\sigma, \sigma+\varepsilon],\mathcal{C}^{\beta})$ such that~(\ref{eq:bspde}) holds, contradicting~(\ref{eq:solu13}). Hence $\sigma=1$ and so $u=\tilde u$ on $[0, 1]$.
\end{proof}

When we take $\beta \in (\frac{1}{3},\frac{1}{2}) $, we can obtain the conclusion in~\cite{Hair11}.
\begin{coro}
Let $\beta \in (\frac{1}{3},\frac{1}{2})$ and $u_0\in \mathcal{C}^{\beta}([0,2\pi],\RR^d)$.
Then, for almost every realization of the driving process $H$, there exists $T>0$ such that equation~(\ref{eq:bspde}) has a unique mild solution taking values in $\mathcal{C}\Big([0, T],\mathcal{C}^{\beta}([0,2\pi],\RR^d)\Big)$. If furthermore $g$ is bounded and all derivatives of $f$ and $g$ are bounded, then this solution is global (i.e., one can choose T arbitrary, independently of $H$).
\mlabel{cor:eu}
\end{coro}

\vskip 0.2in

\noindent
{\bf Acknowledgments.} This work is supported by the National Natural Science Foundation of China (No.~12571019), the Natural Science Foundation of Gansu Province (No.~25JRRA644) and Innovative Fundamental Research Group Project of Gansu Province (No.~23JRRA684). The second author is grateful to Laboratoire de Math\'ematiques Blaise Pascal at  Universit\'e Clermont Auvergne for warm hospitality.

\noindent
{\bf Declaration of interests. } The authors have no conflicts of interest to disclose.

\noindent
{\bf Data availability. } Data sharing is not applicable as no new data were created or analyzed.

\end{document}